\documentclass[12pt,oneside]{amsart}

\usepackage{amsthm,amsfonts,amsbsy,amssymb,amsmath,amscd}
\usepackage{amsmath}
\usepackage[colorinlistoftodos, textsize=tiny]{todonotes}
\def\listtodoname{List of Todos}
\def\listoftodos{\@starttoc{tdo}\listtodoname}

\usepackage[all]{xy}

\usepackage{hyperref} 
\usepackage{soul}
\usepackage{color} 
\usepackage{typearea} 

\usepackage{pdfcomment} 

\usepackage[mathscr]{eucal}
\usepackage{calrsfs}

\usepackage{verbatim}

\usepackage{tikz-cd}

\usepackage{pdfcomment} 

\theoremstyle{plain}
\newtheorem{teo}{Theorem}[section]
\newtheorem{thm}[teo]{Theorem}
\newtheorem{prop}[teo]{Proposition}
\newtheorem{lemma}[teo]{Lemma}
\newtheorem{cor}[teo]{Corollary}

\newtheorem{conj}[teo]{Conjecture}

\newtheorem{thm-defn}[teo]{Theorem-Definition}

\theoremstyle{Definition}
\newtheorem{defn}[teo]{Definition}

\theoremstyle{remark} 
\newtheorem{remark}[teo]{Remark}

\makeatletter
\newcommand{\neutralize}[1]{\expandafter\let\csname c@#1\endcsname\count@}
\makeatother

\numberwithin{equation}{section}

\newcommand{\IQbar}{\overline{\mathbb{Q}}}

\newcommand{\BC}{{\mathbb {C}}}

\newcommand{\BE}{{\mathbb {E}}}

\newcommand{\BG}{{\mathbb {G}}}
\newcommand{\BH}{{\mathbb {H}}}

\newcommand{\BM}{{\mathbb {M}}}
\newcommand{\BN}{{\mathbb {N}}}

\newcommand{\BP}{{\mathbb {P}}}
\newcommand{\BQ}{{\mathbb {Q}}}
\newcommand{\BR}{{\mathbb {R}}}
\newcommand{\BS}{{\mathbb {S}}}

\newcommand{\BV}{{\mathbb {V}}}

\newcommand{\BZ}{{\mathbb {Z}}}

\newcommand{\CB}{{\mathcal {B}}}

\newcommand{\CE}{{\mathcal {E}}}
\newcommand{\CF}{{\mathcal {F}}}

\newcommand{\CJ}{{\mathcal {J}}}

\newcommand{\CL}{{\mathcal {L}}}
\newcommand{\CM}{{\mathcal {M}}}
\newcommand{\CN}{{\mathcal {N}}}
\newcommand{\CO}{{\mathcal {O}}}
\newcommand{\CP}{{\mathcal {P}}}
\newcommand{\CQ}{{\mathcal {Q}}}

\newcommand{\CX}{{\mathcal {X}}}

\newcommand{\CZ}{{\mathcal {Z}}}

\newcommand{\RN}{{\mathrm {N}}}

\newcommand{\sC}{{\mathscr{C}}}
\newcommand{\sD}{{\mathscr{D}}}

\newcommand{\sF}{{\mathscr {F}}}

\newcommand{\sJ}{{\mathscr {J}}}

\newcommand{\sM}{{\mathscr {M}}}

\newcommand{\sO}{{\mathscr {O}}}
\newcommand{\sP}{{\mathscr {P}}}
\newcommand{\sQ}{{\mathscr {Q}}}

\newcommand{\sV}{{\mathscr {V}}}

\newcommand{\AJ}{{\mathrm{AJ}}}

\newcommand{\amp}{{\mathrm{amp}}}

\newcommand{\Aut}{{\mathrm{Aut}}}

\newcommand{\Ce}{{\mathrm{Ce}}}
\newcommand{\Ch}{{\mathrm{Ch}}}

\renewcommand{\div}{{\mathrm{div}}}

\newcommand{\Fal}{{\mathrm{Fal}}}

\newcommand{\GS}{{\mathrm{GS}}}

\newcommand{\Jac}{{\mathrm{Jac}}}

\newcommand{\NT}{{\mathrm{NT}}}

\newcommand{\Pic}{\mathrm{Pic}}

\newcommand{\Sp}{\mathrm{Sp}}

\newcommand{\vol}{{\mathrm{vol}}}

\newcommand{\wt}{\widetilde}
\newcommand{\pp}{\frac{\partial\bar\partial}{\pi i}}

\newcommand{\pair}[1]{\langle {#1} \rangle}

\newcommand\wh{\widehat}

\newcommand{\lra}{{\longrightarrow}}
\newcommand{\iso}{{\overset\sim\lra}}

\renewcommand{\cong}{\simeq}

\begin{document}
\title{Heights of Ceresa and Gross--Schoen cycles}
\author{Ziyang Gao, Shou-Wu Zhang}

\address{Department of Mathematics, UCLA, Los Angeles, CA 90095, USA}
\email{ziyang.gao@math.ucla.edu}

\address{Department of Mathematics, Princeton University, Princeton, NJ 08544, USA}
\email{shouwu@princeton.edu}






\maketitle

{\centering\footnotesize \textit{Dedicated to Gerd Faltings's 70th birthday}\par}

 \begin{abstract}
 We study the Beilinson--Bloch heights of Ceresa and Gross--Schoen cycles in families. We construct that for any  $g\ge 3$, a Zariski open dense subset $\sM_g^\amp$ of $\sM_g$, the coarse moduli of curves of genus $g$ over $\BQ$, such that the heights of Ceresa cycles and Gross--Schoen cycles over $\sM_g^\amp$ have a lower bound and satisfy the Northcott property.
\end{abstract}

\tableofcontents

\section{Introduction}

For a curve over $\BC$, the classical Abel--Jacobi  map is a bijection 
$$\AJ \colon \Pic ^0(C)\iso \Jac (C).$$ When $C$ is defined over a number field $k$,  one can define the N\'eron--Tate height on $\Jac (C)(\bar k)$ using the canonical polarization. This height has a Northcott property, which gives positivity of the height pairing, which in turn is an important ingredient for
 the proof of the Mordell--Weil theorem and the formulation of the Birch and Swinnerton-Dyer conjecture. Notice that  this positivity is also the crucial  part of  the Hodge index theorem of Faltings and Hriljac about arithmetic intersection pairings  on integral models of $C$.
 
 In the 1950s-1960s, Weil and Griffiths extended the Abel--Jacobi map  to homologically trivial cycles (say of codimension $n$) on smooth and projective varieties $X$ over $\BC$ 
 $$\AJ ^n \colon \Ch^n(X)_{\mathrm{hom}}\lra \Jac ^n (X)$$
 where  $\Jac ^n (X)$ are called \textit{intermediate Jacobians} and are only complex tori in general.
 In the 1980s, Beilinson \cite{Be} and Bloch \cite{Bl} independently proposed a conditional definition of heights for these cycles using integral models. 
 They conjectured the positivity of their heights, an extension of the Mordell--Weil theorem, and an extension of the BSD conjecture. However, little is known for any family, and the positivity cannot be proved via a  Northcott property because of the lack of algebraicity of $\Jac ^n (X)$.
 
 \medskip
 This paper aims to study the positivity of the heights of some special cycles constructed by Ceresa in 1983
 and by Gross and Schoen in 1995. The heights of these cycles have important connections to diophantine geometry and special values of $L$-series. 
 We prove the {\it generic positivity} of the Beilinson--Bloch heights for these cycles, including a height inequality and Northcott property.  We will also formulate some general conjectures to extend these results to general families of homologically trivial cycles. 
 
 \subsection{Ceresa  cycles and Gross--Schoen cycles}
 We start with a brief history of the study of Ceresa cycles and Gross--Schoen cycles.
 In 1983, Ceresa \cite{Ceresa} constructed surprisingly simple 1-cycles in the Jacobians of curves that have caught much attention in recent years. More precisely, 
 let $C/k$ be a smooth and proper curve over a field $k$ of genus $g\ge 2$ and $e\in \Ch ^1(C)$ a divisor class of degree $1$.
 Then we have an embedding $i_e \colon C\lra \Jac (C)$ by sending $x\in C$ to the class of $x-e$.  Then the Ceresa cycle $\Ce(e)$ is defined as the 
 cycle $$\Ce (C, e):=\iota_e (C)-[-1]_*\iota _e(C)\in \Ch^{g-1}(\Jac (C)).$$
 This class is homologically trivial. Its class modulo algebraic equivalence does not depend on the choice of $e$. 
 The main result of Ceresa \cite{Ceresa} is that for a general $C$  defined over complex numbers, $\Ce (C, e)$ is not algebraically equivalent to zero if $g(C)\ge 3$. 
 Notice that when $g=2$, or more generally, any hyperelliptic curve,the class $\Ce (C, e)$ is always algebraically trivial because of the identity
 $$\Ce (C, e)=i _e(C)-i _{e'}(C)$$  induced by the hyperelliptic involution $x\mapsto x'$. In 2004, Reed and Hain \cite{HR04}
 studied these cycles in families following Hain's theory biextensions \cite{Ha90}. More precisely, they construct a metrized line bundle $\bar \CB$ on the coarse moduli $\sM_g$ \cite{HR04} and compare it with other topological line bundles on $\sM_g$. As a consequence, they can show that $\Ce(C, e)$ is not torsion if $k$ is the function field of {\em proper curve} in $\sM_g$ over $\BC$.
 
In arithmetic geometry and number theory, it is often preferable to study another $1$-cycle on $C^3$, constructed by Gross and Schoen \cite{GS95}, which turns out to share some common properties with the Ceresa cycles. More precycles, we have 7 partial cycles  $\Delta _I$ in $C^I$  indexed by non-empt subsect $I$ of $\{1, 2, 3\}$. Modified by the base class $e\in \Ch ^1(C)$ we obtain the Gross--Schoen class
\begin{align*}\GS (C, e):=\Delta_{123}&-p_{12}^*\Delta_{12}\cdot p_3^*e-p_{23}^*\Delta_{23}\cdot p_1^*e-p_{13}^*\Delta_{13}\cdot p_2^*e\\
& +p_1^*e\cdot p_2^*e+p_2^*e\cdot p_3^*e+p_1^*e\cdot p_3^* e \in \mathrm{Ch}^2(C^3).
\end{align*}
 where $p_I \colon C^3\lra C^I$ and $p_j \colon C^3\lra C$ are the projections.  The main results of Gross and Schoen \cite{GS95} are that these classes are homologically trivial with well-defined Beilinson--Bloch heights when the base field $k$ is a number field. The main motivation of  in \cite{GS95} is the early paper by 
 Gross and Kudla \cite{Gross-Kudla} in which they propose a conjecture to relate the heights of these cycles to the triple product  $L$-series. 
 The positivity of $\pair{\GS(C, e), \GS(C, e)}_{\mathrm{BB}}$ is then a consequence of the generalized Riemann hypothesis of the triple product $L$-series.
 See \cite{YZZ} for recent progress on this conjecture.  It is known that  the height $\pair{\Ce(e), \Ce(e)}_{\mathrm{BB}}$ is proportional to $\pair{\GS(C, e), \GS(C, e)}_{\mathrm{BB}}$ with the ratio a positive constant depending only on $g(C)$, and that $\Ce(e)$ in torsion in $\mathrm{Ch}^2(C^3)$ if and only if $\GS(C,e)$ is torsion in $\mathrm{Ch}^{g-1}(\mathrm{Jac}(C))$; see \cite{Zh10}.
 
 In 2010, when the base field $k$ is a number field,  the second-named author in \cite{Zh10} proved an identity between 
 heights $\pair{\GS(C, e), \GS(C, e)}$  of  Gross--Schoen cycles, and the self-intersection $\overline{\omega} _{\mathrm{a}}^2$ of the admissible dualising sheaf of $C$:
 $$\pair{\GS(C, e), \GS(C, e)}_{\mathrm{BB}}=\frac {2g+1}{2g-2}\overline{\omega} _{\mathrm{a}}^2+6(g-1)\pair{e-\xi, e-\xi}_\NT-\sum_v \varphi (C_v)\deg v$$
 where $\xi\in \Pic ^1(C)$ such that $(2g-2)\xi=\omega_C$, the canonical class of $C$, and $\varphi(C_v)$ are some local invariants $C_v:=C\otimes K_v$ at places $v$ of $K$. 
Thus $\pair{\GS(C, e), \GS(C, e)}_{\mathrm{BB}}$ reaches its minimum when $e=\xi$. 
 A major application of  the conjectured positivity of $\pair{\GS(C, \xi), \GS(C, \xi)}_{\mathrm{BB}}$ is that it will give an effective version of the Bogomolov conjecture proved by  Ullmo \cite{Ull98}. 
 This idea  has led to some unconditional effective Bogomolov conjectures by R. de Jong \cite{dJ18a} and Wilm \cite{Wil},
 and Yuan's new proof \cite{YuanArithBig} of a uniform Bogomolov result (with generalization) originally obtained by Dimitrov--Gao--Habegger \cite{DGH} and K\"{u}hne \cite{Kuhne}.
 
In \cite{Zh10}, the Northcott property was established for $\pair{\GS(C, \xi), \GS(C, \xi)}_{\mathrm{BB}}$ for $C$ in {\it proper} subvarieties of $\sM_g$. A main goal of this paper is to extend this Northcott property to {\it non-empty open} subvarieties of $\sM_g$ when $g\ge 3$. As a consequence this proves the generic positivity for $\pair{\GS(C, \xi), \GS(C, \xi)}_{\mathrm{BB}}$.
 
\subsection{Main results}\label{SubsectionMainResultsGSCe}
Our main result is the generic positivity and Northcott property for the Beilinson--Bloch heights of the Gross--Schoen cycles and Cerasa cycles parametrized by moduli of curves. 
 These are the simplest situations where the Beilinson--Bloch heights are unconditionally defined \cite{GS95, Ku}.

Let $g\ge 3$ and  $\sM_g$ be the  moduli space  of the pairs $(C, \xi)$ of a smooth projective curve $C$ of genus $g$ and a class $\xi\in \Pic ^1(C)$ 
so that $(2g-2)\xi=\omega_C$, the canconial class of $C$. Over $\BZ[1/(2g-2)]$, this is an \'etale cover the moduli stack $\BM_g$ of curves of genus $g$ of degree 
$(2g-2)^{2g}$. Let $\sC_g/\sM_g$ be the universal curve, and let 
 $\pi \colon \Pic^1(\sC_g/\sM_g)\lra \sM_g$ be  the relative Picard group of degree $1$ on $\sC_g/\sM_g$. 
 Thus $\Pic^1(\sC/\sM)$ paremetrizes the  triples $(C, \xi, e)$ with $(C, \xi)\in \sM_g$ and $e\in\Pic ^1(C)$.
 
For each $s \in \sM_g(\BC)$ parametrizes  $(C_s, \xi _s)$, let $\GS (s)$ and $\Ce (s)$ denote the corresponding cycles $\GS (C_s, s)$ and $\Ce (C_s, s)$ respectively.  
For each $s\in \Pic ^1(\sC/\sM_g)$ paramteizes the triple $(C_s, s, e_s)$, let  $\GS (s)$ and $\Ce (s)$ denote the cycles  $\GS (C_s, e_s)$ and $\Ce(C_s, e_s)$ respectively.


\begin{thm}\label{MainTheorem}
For each $g\ge 3$, there exist a Zariski open dense subset $U$ of $\sM_g$ defined over $\BQ$ and  positive numbers $\epsilon$ and $c$, such that for any $s\in \pi^{-1}(U)(\IQbar)$, we have 
\begin{equation}\label{EqGSComparison}
\pair{\GS(s),  \GS(s)}_{\mathrm{BB}}\ge \epsilon (h_\Fal (C_s)+h_{\NT} (e_s-\xi_s))-c,
\end{equation}
$$\pair{\Ce(s),  \Ce(s)}_{\mathrm{BB}}\ge \epsilon (h_\Fal(C_s)+h_{\NT} (e_s-\xi_s))-c,$$
where $h_\Fal$ is the Faltings height of $C_s$,  and $h_\NT$ is the canonical N\'eron--Tate height on $\Jac (\sC_s)(\IQbar)$.
\end{thm}

Theorem~\ref{MainTheorem} yields the following corollary immediately.
\begin{cor}\label{CorNorthcott}
For each $g\ge 3$ and the Zariski open dense subset $U$ of the $\sM_g$ defined over $\BQ$ as in Theorem~\ref{MainTheorem}, we have 
\begin{itemize}
\item[-] (lower bound) There exists a number $B=B(g)$ such that 
\[
\pair{\GS(s),  \GS(s)}_{\mathrm{BB}} > B\qquad\text{ and }\qquad \pair{\Ce(s),  \Ce(s)}_{\mathrm{BB}} > B
\]
 for all $s\in \pi^{-1}(U) (\IQbar)$.
\item[-] (Northcott property) For any $H, D\in \BR$, we have
\begin{align*}
\qquad\qquad  \#\{s \in \pi^{-1}(U) (\IQbar) : \quad \deg [\BQ(s):\BQ]<D, \quad \pair{\GS(s),  \GS(s)}_{\mathrm{BB}}<H\} < \infty, \\
\qquad \#\{s\in \pi ^{-1}(U) (\IQbar) : \quad \deg [\BQ(s):\BQ]<D, \quad \pair{\Ce(s),  \Ce(s)}_{\mathrm{BB}}<H\}<\infty.
\end{align*}
In particular for each number field $K$, there exist at most finitely many $s \in  \pi^{-1}(U)(K)$ such that $\pair{\GS(s),  \GS(s)}_{\mathrm{BB}}$ or 
$\pair{\Ce(s),  \Ce(s)}_{\mathrm{BB}}$ is negative.
\end{itemize}
\end{cor}

Before our work, de Jong \cite{dJ18}  constructed a family of genus $3$ curves with unbounded Beilinson--Bloch height, and proved some results over function fields with Shokrieh \cite{dJS}.

\begin{remark}\label{RemarkMainTheorem}
In practice, we wish to determine the subset $U$ in Theorem~\ref{MainTheorem}. If \eqref{EqGSComparison} holds over Zariski open dense subsets $U_1$ and $U_2$ of $\sM_g$, 
then it also holds over $U_1\cup U_2$. Thus, there is the maximal open subset of $\sM_g$ over $\BQ$ so that \eqref{EqGSComparison} holds. 
We call this maximal open subset the {\it ample locus} and denote it by $\sM_g^{\mathrm{amp}}$. 
We shall prove that its complement $\sM_g \setminus \sM_g^{\mathrm{amp}}$ is precisely the Betti stratum. In particular $\mathrm{GS}(s)$ and $\mathrm{Ce}(s)$ are non-torsion for every non-$\IQbar$ point $s$ in $\sM_g^{\mathrm{amp}}$. 
See  Theorem~\ref{ThmDescriptionAmpleLocus}.
\end{remark}

\subsection{Sketch of proof}

The first key ingredient of the proof is the theory of adelic line bundles for quasi-projective varieties \cite{YZ21}. 

In our paper, we need to  construct a suitable adelic line bundle $\overline{\CL}$ on $\sM_g$ which defines the Beilinson--Bloch height on $\sM_g$ and {\it whose curvature form is semi-positive}.   More precisely, we prove:
\begin{thm}[Theorem~\ref{ThmZhangBB} and \ref{PropTwoCurvatureEqualAdeGS}]\label{ThmZhangBBIntro}
There exists an integrable $\BQ$-adelic line bundle $\overline{\CL}$ on $\sM_g$ such that $h_{\overline{\CL}}(s) = \pair{\GS(\xi_s),  \GS(\xi_s)}_{\mathrm{BB}}$ for all $s \in \sM_g(\IQbar)$. Moreover, its curvature form $c_1(\overline{\CL})$ is semi-positive.
\end{thm}
A construction of such an $\overline{\CL}$ was obtained by Yuan \cite[$\S$3.3.4]{YuanArithBig} without proving $c_1(\overline{\CL})\ge 0$. Our construction is different. We use the canonical polarized dynamical system on $\mathrm{Jac}(\sC_g/\sM_g)$, and $\overline{\CL}$ is ultimately obtained from a suitable Deligne pairing. However, in both constructions one cannot obtain nefness condition on $\overline{\CL}$. This is a major obstacle to study the family of Gross--Schoen cycles compared to the N\'{e}ron--Tate family. Showing that our $h_{\overline{\CL}}$ is precisely the Beilinson--Bloch height uses the computation from the second-named author's \cite[Thm.~2.3.5]{Zh10}. The semi-positivity of $c_1(\overline{\CL})$ is established by relating it to the Betti form, which we will explain later.

With Theorem~\ref{ThmZhangBBIntro} in hand, the desired height comparison \eqref{EqGSComparison} in Theorem~\ref{MainTheorem} (with $U$ defined over $\IQbar$ instead of $\BQ$) is equivalent to the bigness of the generic fiber (sometimes called the geometric part) $\wt{\CL}$ of $\overline{\CL}$, \textit{i.e.} to show that $\widehat{\vol}(\wt{\CL})>0$. 

A main theorem on the arithmetic part is the following {\it volume identity}.
\begin{thm}[Theorem~\ref{PropHSforGS}]\label{ThmVolumeIdentityIntro}
We have
\begin{equation}\label{EqVolumeIdentityIntro}
\widehat{\vol}(\wt{\CL}) = \int_{\sM_g(\BC)} c_1(\overline{\CL})^{\wedge \dim \sM_g}.
\end{equation}
\end{thm} 
In fact, we prove this volume identity for any subvariety $S$ of $\sM_{g,\BC}$. This stronger version is needed to descend $U$ from $\IQbar$ to $\BQ$ in Theorem~\ref{MainTheorem}; see Corollary~\ref{cor-rationality}.

It is known \cite[Lem.~5.4.4]{YZ21} that the right hand side of the volume identity equals the top self-intersection number of $\wt{\CL}$. Hence the volume identity is in the flavor of the arithmetic Hilbert--Samuel theorem. However, arithmetic Hilbert--Samuel is not applicable in our situation because we do not  know how to prove the nefness $\wt{\CL}$ (for definition of nefness see \cite[$\S$2.5.2 and 2.5.3]{YZ21}). Instead, we construct a sequence of model Hermitian line bundles $\overline{\CL}_i$ (on some projective arithmetic models of $\sM_g$) converging to $\overline{\CL}$ and prove the volume identity for each $\overline{\CL}_i$. The proof relies heavily our concrete construction of $\overline{\CL}$, and a crucial ingredient is Demailly's Morse Inequality \cite{De91}.

Hain's works on the 
 Archimedean local height pairing and the Betti form are the key ingredients to 
 bridging the arithmetic arguments above to a geometric statement. 
Indeed, he proved that the archimedean local height pairing could be computed using the metrized biextension line bundle \cite{Ha90}, and that the curvature form of the metrized biextension line bundle, called the {\em Betti form}, is a semi-positive $(1,1)$-form \cite{HR04, Ha13}.

In our case of the Gross--Schoen cycles, 
we show in Proposition~\ref{PropTwoCurvatureEqualAdeGS} that $c_1(\overline{\CL})$ is precisely the Betti form $\beta_{\GS}$. This result can also be deduced from R.~de Jong's \cite[(9.3)]{dJ16}. Thus, roughly speaking by the volume identity \eqref{EqVolumeIdentityIntro} and its stronger form, we are left to prove that   the volume form 
$\beta_{\GS}^{\wedge \dim \sM_g}$ {\it does not vanish identically, and its vanishing locus is Zariski closed} subset of $\sM_{g,\BC}$. More precisely for any subvariety $S \subseteq \sM_{g,\BC}$, we define the {\it Betti stratum} $S^{\mathrm{Betti}}(1)$ (see above Lemma~\ref{LemmaBettiFoliationBettiRank}) which is a real analytic subset of $S$ satisfying the following property: $\beta_{\GS}|_{S}^{\wedge \dim S} \equiv 0$  if and only if $S^{\mathrm{Betti}}(1) = S$. We prove:

\begin{thm}[Theorem~\ref{ThmBettiRankFormulaMain} and \ref{ThmZariskiClosedDegLoci}]\label{ThmBettiRankIntro}
For $g \ge 3$, $\beta_{\GS}^{\wedge \dim \sM_g}$ does not vanish identically. 

Moreover for any subvariety $S \subseteq \sM_{g,\BC}$, $S^{\mathrm{Betti}}(1)$ is Zariski closed in $S$.
\end{thm}
Hain \cite{Ha24} also proved $\beta_{\GS}^{\wedge \dim \sM_g} \not\equiv 0$ for $g\ge 3$ in a different way without showing the Zariski closedness of the Betti stratum $\sM_g^{\mathrm{Betti}}(1)$.  
Our approach can in fact be applied to any subvariety $S \subseteq \sM_g$ to get a checkable criterion \eqref{EqCritBettiRank} for $\beta_{\GS}|_S^{\wedge \dim S} \equiv 0$, and more generally to any family (see Theorem~\ref{ThmBettiRankGeneralFamily} and \ref{ThmBettiStrataZarClosedGeneralFamily}). Our proof  of Theorem~\ref{ThmBettiRankIntro}
is inspired by the first-named author's \cite{GaoBettiRank}, and the key ingredients are the mixed Ax--Schanuel theorem \cite{KennethChiuAS, GKAS} and the o-minimal structure associated with the period map \cite{BBKT}.

\subsection{Organization of the paper}
In \S\ref{SectionArchimedeanLocalHeight}, we review  Hain's work on Archimedean local height pairings for a general family of homologically trivial cycles. We then focus on the Gross--Schoen and Ceresa cycles in  \S\ref{SectionGSCe}. 
In  \S\ref{SectionArithmetic} and  \S\ref{SectionVolIdentity},  we explain how to construct the desired adelic line bundle and prove the semipositivity Theorem \ref{ThmZhangBBIntro}, 
and prove the volume identity  Theorem \ref{ThmVolumeIdentityIntro}. 

We then recall in $\S$\ref{SectionPeriodMap} the definition of the period map associated with a normal function  and explain how to see the Betti foliation in this language.  With these preparations we prove Theorem~\ref{ThmBettiRankIntro} in $\S$\ref{SectionBettiRank}, which is the core of the geometric part of our proof.

We then prove the main results of our paper in $\S$\ref{SubsectionBignessProof}, where we prove Theorem~\ref{MainTheorem} and the claims in Remark~\ref{RemarkMainTheorem}.
In $\S$\ref{SectionFuture}, we explain our project on studying the positivity of Beilinson--Bloch heights for other families.

To make our paper more self-contained and for readers' convenience, we include an Appendix~\ref{AppVMHS} which collects some definitions about mixed Hodge structures and their variations. 

\subsection*{Acknowledgements}
We would like to thank Ari Shnidman for asking us the question about the Northcott property for Ceresa cycles, 
Richard Hain for numerous inspiring discussions and for sharing his notes, Xinyi Yuan for his help during the preparation of this project and comments on a previous version of the paper, Dick Gross for suggesting we look at $\Pic^1(\sC_g/\sM_g)$ instead of just $\sM_g$, and Peter Sarnak for asking the question about the positivity of the Beilinson--Bloch heights. 
Part of this work was completed when ZG visited SZ at Princeton University in January, February, and May 2024. ZG would like to thank Princeton University for its hospitality. 


Ziyang Gao received funding from the European Research Council (ERC) under the European Union’s Horizon 2020 research and innovation program (grant agreement n$^{\circ}$ 945714). Shouwu Zhang is supported by an NSF award, DMS-2101787.

\section{Archimedean local height pairing and metrized Poincar\'{e} bundle}\label{SectionArchimedeanLocalHeight}
In this section, we recall some results of Hain on the archimedean local heights, which works for any family of homologically trivial cycles.

Let $S$ be a smooth irreducible quasi-projective variety.  Let $f \colon X \rightarrow S$ be a smooth projective morphism of algebraic varieties with irreducible fibers of dimension $d$, defined over $\BC$.
Let $p\le d$ be a non-negative integer. Let $Z$ be a family of homologically trivial cycles of codimension $p$ in $X/S$. Namely $Z$ is a formal sum of integral subschemes of $X$ which are flat and dominant over $S$ such that each fiber $Z_s$ is a homologically trivial cocycle of codimension $p$ in $X_s$, \textit{i.e.} $Z_s$ is in the kernel of the cycle map 
$$Z^p(X_s)  \rightarrow H^{2p}(X_s,\BZ(p)).$$

\subsection{Height pairing at Archimedean places}\label{SubsectionHeightPairingPoincare}
Let $q:=d+1-p$. 
Let  $W$ be a family of homologically trivial cycles of codimension $q$ in $X/S$. Assume $Z$ and $W$ have disjoint supports over the generic fiber, then up to replacing $S$ by a Zariski open dense subset, we may assume that  $Z_s$ and $W_s$ have disjoint supports for all $s \in S(\BC)$.

In Arakelov's theory, the local height pairing  at archimedean places in this context is defined by 
\begin{equation}\label{EqLocalHeightPairing}
\langle Z_s, W_s \rangle_{\infty} = \int_{X_s} g_{W_s}\delta_{Z_s}.
\end{equation}
where $g_{W_s}$ is a Green's current for $W_s$, namely a current satisfying the equation:
$$\pp g_{W_s}=\delta _{W_s}.$$
This pairing is symmetric.

Hain \cite[$\S$3.3]{Ha90} related the local height pairing at archimedean places \eqref{EqLocalHeightPairing} to the metrized Poincar\'{e} bundle. We will review this in the next subsection.




\subsection{Intermediate Jacobian and the normal function associated with $Z$}\label{SubsectionANF}
 For each $s \in S(\BC)$, the cohomology group $H^{2p-1}(X_s,\BZ)$ is endowed with a natural Hodge structure of weight $2p-1$, with the Hodge filtration 
 $$H^{2p-1}(X_s,\BC) \supsetneq \cdots \supsetneq F^p_s \supsetneq F^{p+1}_s \supsetneq \cdots \supsetneq 0$$
  where $F^m_s$ is the $\BC$-subspace of $H^{2p-1}(X_s,\BC) \cong H^{2p-1}_{\mathrm{dR}}(X_s)$ generated by the $(r,r')$-forms with $r \ge m$. This makes $R^{2p-1} f_* \BZ_{X}$ into a variation of Hodge structures on $S$ of weight $2p-1$.  
Applying the Tate twist we obtain a variation of Hodge structures 
\[
\BV_{\BZ}:=R^{2p-1}f_*\BZ(p)_{X}
\]
 of weight $-1$. Denote by $\sF^{\bullet}$ the Hodge filtration; it is the decreasing chain 
 $$\cdots \supsetneq \sF^0  \supsetneq \sF^1  \supsetneq \cdots$$ of holomorphic subbundles of the vector bundle $\sV := \BV_{\BZ} \otimes_{\BZ}\sO_S$ such that the fiber $\sF^m_s$ is the Tate twist of $F^{m+p}_s$ for each  $s \in S(\BC)$.

The {\it $p$-th intermediate Jacobian} of $X \rightarrow S$ is defined to be
\begin{equation}\label{EqIntermediateJacobianPth}
\sJ^p(X/S) := \BV_{\BZ} \backslash \sV / \sF^0.
\end{equation}
Then $\sJ^p(X/S) \rightarrow S$ is a holomorphic family of compact complex tori. If $S$ is a point, then we simply write $\sJ^p(X)$.

For each $s \in S(\BC)$, we have  $\sJ^p(X/S)_s = \sJ^p(X_s)$ and Griffiths \cite[$\mathsection$11]{GriffithsAJ} constructed the  {\em Abel--Jacobi map} 
\begin{equation}\label{EqAJmap}
\AJ \colon \mathrm{Ch}^p(X_s)_{\hom} \longrightarrow \sJ^p(X_s)
\end{equation}
where $\mathrm{Ch}^p(X_s)_{\hom}$ is the kernel of the cycle map $\mathrm{Ch}^p(X_s)  \rightarrow H^{2p}(X_s,\BZ)$. 
Hence the family $Z$  induces a holomorphic section
\begin{equation}\label{EqNormalFunctionFromFamily}
\nu_{Z} \colon S \rightarrow \sJ^p(X/S), \qquad s \mapsto \AJ(Z_s). 
\end{equation}
We call this $\nu_Z$ the {\it admissible normal function} associated with $Z$.  

\subsection{Betti foliation}\label{SubsectionBettiFoliation1}
A key object is the {\it Betti foliation} $\CF_{\mathrm{Betti}}$ on $\sJ^p(X/S)$ defined as follows. The fiberwise  isomorphism $\BV_{\BR,s} \iso  \BV_{\BC,s}/\sF^0_s = \sV_s/\sF^0_s$ makes  
$\sJ^p(X/S)$ into a local system of real tori
\begin{equation}\label{EqLocSysTori}
\BV_\BR/\BV_\BZ \xrightarrow{\sim} \sJ^p(X/S).
\end{equation}
The induced foliation is the Betti foliation $\CF_{\mathrm{Betti}}$. More precisely, for any point $x\in \sJ^p(X/S)$, there is a local section  $\sigma \colon U \lra \sJ^p(X/S)$ 
from a neighborhood $U$ of $\pi(x)$ in $S^{\mathrm{an}}$, with $x \in \sigma(U)$,  represented by a section of the local system $\BV_\BR$. 
 The manifolds $\sigma (U)$ gluing together to a foliation $\CF_{\mathrm{Betti}}$ on $\sJ^p(X/S)$.  In particular, all torsion multi-sections are leaves of $\CF_{\mathrm{Betti}}$.

\subsection{Metrized Poincar\'{e} bundle and metrized tautological bundle}
The torus fibrations $\sJ^p(X/S)$ and $\sJ^q(X/S)$ are dual to each other by Poincar\'{e} duality. 
The general theory of biextension says that there exists a unique line bundle $\sP \rightarrow \sJ^p(X/S) \times_S\sJ^q(X/S)$ 
with a rigidification 
 $$\epsilon^*\sP \cong \sO_S$$
  for the zero section $\epsilon$ of $ \sJ^p(X/S) \times_S\sJ^q(X/S) \rightarrow S$, such that 
 over each $s \in S(\BC)$, $\sP|_{\sJ^p(X_s) \times \{\lambda\}}$ represents $\lambda \in \mathrm{Pic}^0(\sJ^p(X_s)) = \sJ^p(X_s)^\vee = \sJ^q(X_s)$.

Moreover, $\sP$ can be endowed a canonical Hermitian metric $\| \cdot \|_{\mathrm{can}}$ uniquely determined by the following properties:
\begin{enumerate}
\item   the curvature of $\overline{\sP}:=(\sP, \| \cdot \|_{\mathrm{can}})$ is translation invariant on each fiber of $ \sJ^p(X_s) \times_S\sJ^q(X_s)$; 
\item  $\epsilon^* \overline{\sP} \cong (\sO_S, \| \cdot \|_{\mathrm{triv}})$ for the trivial metric on $\sO_S$. 
\end{enumerate}

\begin{defn}\label{DefnMetrizedPoincare1}
The metrized line bundle $\overline{\sP} := (\sP, \| \cdot \|_{\mathrm{can}})$ is called the {\em metrized Poincar\'{e}  bundle}.
\end{defn}

Let $Z$ and $W$ be family of cycles on $X$ over $S$ or codimension $p$ and $q$ respectively. Then  we obtain a metrized line bundle $(\nu_Z,\nu_W)^*\overline{\sP}$ on $S$. Use $\| \cdot \|$ to denote this induced metric on $(\nu_Z,\nu_W)^*\sP$.

 When $Z, W$ are disjoint over the generic point $s$ of $S$, Hain \cite[Prop.~3.3.2]{Ha90} constructed a non-zero section $\beta_{Z,W}$ of the line bundle $(\nu_Z,\nu_W)^*\sP \rightarrow S$ in view of the Hodge-theoretic construction of $\sP$ in \cite[$\S$3.2]{Ha90}, and  proved \cite[Prop.~3.3.12]{Ha90}  $\log \| \beta_{Z,W}(s) \| = -\int_{X_s} g_{W_s}\delta _{Z_s}$ for all $s \in S(\BC)$. To summarize, we have
\begin{prop}[Hain]\label{PropHainLocalHeight}
We have $- \log \| \beta_{Z,W}(s) \| = \langle Z_s, W_s \rangle_{\infty}$ for all $s\in S(\BC)$.
\end{prop}

For our purpose we are interested in the case $p=q=\frac{d+1}{2}$ and $Z$ and $W$ are rationally equivalent. Thus $\nu_Z = \nu_W$, which we denote by $\nu$ for ease of notation. Then $(\nu,\nu)^*\overline{\sP} = \nu^*\Delta^*\overline{\sP}$ for the diagonal $\Delta\colon  \sJ^p(X/S) \rightarrow \sJ^p(X/S)\times_S \sJ^p(X/S)$.

\begin{defn}\label{DefnMetrizedPoincare}
The metrized line bundle $\overline{\sP^{\Delta}} := \Delta^*\overline{\sP}$ is called the {\em metrized tautological bundle} on  $\sJ^p(X/S)$.
\end{defn}

\subsection{Betti form and Betti rank}\label{SubsectionDiagPoincareNormalFunction}
The Betti form is the curvature form of $\nu^*\overline{\sP^{\Delta}}$, which is a semi-positive $(1,1$)-form. Let us review some details due to Hain.

Denote by $\BV := \BV_{\BZ} \otimes_{\BZ_S}\BQ_S$. 
The Hodge structure $\BV_{\BZ,s} = H^{2p-1}(X_s,\BZ(p))$ is polarizable by the hard Lefschetz theorem and the Lefschetz decomposition, \textit{i.e.} there exists an anti-symmetric bilinear map $Q \colon \BV_s \times \BV_s \rightarrow \BQ$ satisfying:
\begin{enumerate}
\item[(i)]$Q(\BV_s^{r,r'}, \overline{\BV_s^{t,t'}}) = 0$ unless $r=t$ and $r'=t'$; 
\item[(ii)] $\sqrt{-1}^{r'-r} Q(u,\bar{u}) > 0$ for all non-zero $u \in \BV_s^{r,r'}$. 
\end{enumerate}
Moreover, we have a polarization of the variation of Hodge structures $\BV_{\BZ}$, \textit{i.e.} a map of local systems
\[
\sQ \colon \BV \times_S \BV \rightarrow \BQ_S
\]
such that each fiber $\sQ_s$ is the $Q$ above. By abuse of notation, we use $Q$ to denote $\sQ_s$ for any $s \in S(\BC)$.

The curvature form $c_1(\overline{\sP^{\Delta}})$ is characterized in \cite[Prop.~7.1 and 7.3]{HR04}. Here is an explicit formula, using the Betti foliation $\CF_{\mathrm{Betti}}$ on $\sJ^p(X/S)$. 
Write $\pi \colon \sJ^p(X/S) \rightarrow S$ for the natural projection. The Betti foliation gives a decomposition of the holomorphic tangent space $T_x \sJ^p(X/S)= T_x \CF_{\mathrm{Betti}} \oplus T_x \sJ^p(X_{\pi(x)})$ at each $x \in \sJ^p(X/S)$, and set 
\[
q_x \colon T_x \sJ^p(X/S)  ~ \lra ~ T_x \sJ^p(X_{\pi(x)}) = T_x \sJ^p(X/S)_s
\]
to be the projection to the second factor. 
Under the canonical identification of $T_x \sJ^p(X/S)_s$ with $\BV_{\BR,s}$, the $2$-form $c_1(\overline{\sP^{\Delta}})$ is defined by
\begin{equation}\label{EqBettiFormFormula}
c_1(\overline{\sP^{\Delta}})(v_1,v_2) = 2 Q(q_x(v_1), q_x(v_2)), \qquad \text{ for any }v_1, v_2 \in  T_x \sJ^p(X/S).
\end{equation}

\vskip 0.3em
Next we turn to $\nu^*\overline{\sP^{\Delta}}$.  Hain proved \cite[Thm.~13.1]{Ha13} that $\beta_{\nu} := \nu^*c_1(\overline{\sP^{\Delta}})$ is a semi-positive $(1,1)$-form.

\begin{defn}\label{DefnBettiForm}
This semi-positive $(1,1)$-form $\beta_{\nu}$ is called the {\em Betti form associated with $\nu$}.
\end{defn}

In the proof of \cite[Thm.~13.1]{Ha13}, the following explicit formula for $\beta_{\nu}$ is derived from \eqref{EqBettiFormFormula}. 
For each $s \in S(\BC)$, we have a $\BC$-linear map
\[
\nu_{\mathrm{Betti},s} \colon T_s S \xrightarrow{\mathrm{d}\nu} T_{\nu(s)} \sJ(\BV_{\BZ})  \xrightarrow{q_{\nu(s)}} T_{\nu(s)}\sJ(\BV_{\BZ,s}) = T_{\nu(s)} \sJ^p(X_s).
\]
We have  $T_{\nu(s)} \sJ^p(X_s) = \BV_{\BC,s} / \sF^0_s$. For any $u \in T_s S$, \textit{Griffiths' transversality} yields $\nu_{\mathrm{Betti},s}(u) \in \sF^{-1}_s / \sF^0_s$, which is canonically isomorphic to $\BV_s^{-1,0}$. Then 
\[
\beta_{\nu}(u,\bar{u}) = 2\sqrt{-1}Q( \nu_{\mathrm{Betti},s}(u), \overline{\nu_{\mathrm{Betti},s}(u)})\qquad \text{ for any } u \in T_sS,
\]
 and hence 
\begin{equation}\label{EqBettiFormSemiPositive}
\beta_{\nu}(u,\bar{u}) \ge 0 \text{ for all } u \in  T_s S, \text{ with equality if and only if } \nu_{\mathrm{Betti},s}(u) = 0.
\end{equation}
Notice that this also explains the semi-positivity of $\beta_{\nu}$. 
Now we use \eqref{EqBettiFormSemiPositive} to prove the following proposition. 
Define the {\it Betti rank}
\begin{equation}\label{EqDefnBettiRank}
r(\nu) := \max_{s \in S(\BC)} \dim\nu_{\mathrm{Betti},s}(T_s S)
\end{equation}
and the {\it Betti stratum}
\begin{equation}\label{EqDefnBettiStratum}
S^{\mathrm{Betti}}(1) := \{s \in S(\BC): \dim_{\nu(s)}(\nu(S)\cap \CF_{\mathrm{Betti}}) \ge 1\}.
\end{equation}

\begin{prop}\label{CorBettiFormBettiRank}
The following are equivalent:
\[
\text{\rm(i) } \beta_{\nu}^{\wedge \dim S} \equiv 0\text{;} \qquad \text{\rm (ii) } r(\nu) < \dim S\text{;} \qquad \text{\rm (iii) }S^{\mathrm{Betti}}(1) = S.
\]
\end{prop}
\begin{proof}
The equivalence of (ii) and (iii) is easy to show. Let us prove the equivalence of (i) and (ii).

Assume (i) holds. Then for any $s \in S(\BC)$, there exists $0\not= u \in  T_s S$ with $\nu_{\mathrm{Betti},s}(u) = 0$ by \eqref{EqBettiFormSemiPositive}. So $\ker\nu_{\mathrm{Betti},s}\not=0$, and hence $\dim \nu_{\mathrm{Betti},s}(T_s S) < \dim S$. So (ii) holds.

Assume (ii) holds. Then for any $s \in S(\BC)$, 
 there exists $0\not= u \in \ker\nu_{\mathrm{Betti},s}$.  Hence $\beta_{\nu}(u,\bar{u}) = 0$ by \eqref{EqBettiFormSemiPositive}. So $u$ is an eigenvector of the Hermitian matrix defining $\beta_{\nu}$ with eigenvalue $0$. So the determinant of this matrix is $0$, so $\beta_{\nu}^{\wedge \dim S} = 0$ at $s$. 
So (i) holds.
\end{proof}

\section{Gross--Schoen and Ceresa cycles}\label{SectionGSCe}
In this The goal of this section is to study the normal functions for the Gross--Schoen and Ceresa cycles. The goal is to show that the images of these normal functions lie is the relative intermediate Jacobian of an irreducible variation of Hodge structures.

Let $g\ge 3$ and  $\sM_g$ be the  moduli space of the pairs $(C, \xi)$ of a smooth projective curve $C$ of genus $g$ and a class $\xi\in \Pic ^1(C)$ 
so that $(2g-2)\xi=\omega_C$, the canconial class of $C$. Let $f \colon \sC_g \rightarrow \sM_g$ be the universal curve. Denote by $\sC_g^{[3]} := \sC_g\times_{\sM_g}  \sC_g\times_{\sM_g} \sC_g$ and by $\sJ_g = \mathrm{Jac}(\sC_g/\sM_g)$ the relative Jacobian. Then we have two normal functions as particular cases of the construction \eqref{EqNormalFunctionFromFamily} for Gross--Schoen and Ceresa cycles:
\begin{align}\label{EqNFGSCe}
\nu_{\GS}  \colon \sM_g \longrightarrow \sJ^2(\sC_g^{[3]}/\sM_g), & \qquad s \mapsto \AJ(\GS(s)), \\
 \nu_{\Ce} \colon \sM_g \longrightarrow \sJ^{g-1}(\sJ_g/\sM_g), & \qquad s \mapsto \AJ(\Ce(s)).  \nonumber
\end{align}

\subsection{More general intermediate Jacobians}\label{SubsectionBettiFoliation}
In $\mathsection$\ref{SubsectionANF} around \eqref{EqIntermediateJacobianPth}, we constructed the $p$-th intermediate Jacobian $\sJ^p(X/S)$ of the variation of Hodge structures $R^{2p-1}f_*\BZ_{X}(p)$ of weight $-1$ for any family of projective varieties $f\colon X \rightarrow S$, and defined the Betti foliation. This construction can be generalized to an arbitrary  variation of Hodge structures $\BV_{\BZ}$ of weight $-1$  on $S$ (so that each fiber $\BV_{\BZ,s}$ is a Hodge structure of weight $-1$), with Hodge filtration $\sV \supsetneq \cdots \supsetneq \sF^0 \supsetneq \sF^1 \supsetneq \cdots \supsetneq 0$ which is a decreasing chain of holomorphic sub-vector bundles of $\sV:=\BV_{\BZ} \otimes_{\BZ_S} \sO_S$. The {\it intermediate Jacobian} is defined to be
\begin{equation}
\sJ(\BV_{\BZ}) := \BV_{\BZ} \backslash \sV/\sF^0
\end{equation}
which is again a family of compact complex tori on $S$. In the geometric situation \eqref{EqIntermediateJacobianPth}, we have $\BV_{\BZ} = R^{2p-1}f_*\BZ_{X}(p)$. Again as in $\mathsection$\ref{SubsectionBettiFoliation1}, we have the {\it Betti foliation} $\CF_{\mathrm{Betti}}$ on $\sJ(\BV_{\BZ})$ induced by $\sJ(\BV_{\BZ}) \cong \BV_{\BR}/\BV_{\BZ}$, which arises from the  fiberwise isomorphism $\BV_{\BR,s} \iso  \BV_{\BC,s}/\sF^0_s = \sV_s/\sF^0_s$.

\subsection{Passing to sub-VHS} 
We start the discussion on each individual curve $C$. 
Let $s=(C, \xi)\in \sM_g (\BC)$ as above.
First we have 
$$\AJ(\GS(s)) \in \sJ^2(C^3) = \sJ(H^3(C^3,\BZ)(2)), $$
$$\AJ(\Ce(s)) \in \sJ^{g-1}(\mathrm{Jac}(C)) = \sJ(H^{2g-3}(\mathrm{Jac}(C),\BZ)(g-1)).$$ 
Using the Poincar\'{e} duality on $\mathrm{Jac}(C)$ and on $C$, we get 
\[
H^{2g-3}(\mathrm{Jac}(C),\BZ)(g-1) = H_3(\mathrm{Jac}(C),\BZ)(-1) = \bigwedge\nolimits^3 H_1(C,\BZ)(-1) = \bigwedge\nolimits^3 H^1(C,\BZ)(2).
\]
 So $\AJ(\Ce(s)) \in \sJ(\bigwedge^3 H^1(C,\BZ)(2)).$

We have $\bigwedge^3 H^1(C,\BZ)(2) \subseteq H^3(C^3,\BZ)(2)$ in the following way. First, the K\"{u}nneth formula gives a decomposition $ H^3(C^3,\BZ) = H^1(C,\BZ)^{\otimes 3} \bigoplus H^1(C,\BZ)(-1)^{\oplus 6}$. Next, $H^1(C,\BZ)^{\otimes 3}$, as a subspace of $H^3(C^3,\BZ)$, has a basis consisting of $\alpha_1 \wedge \alpha_2 \wedge \alpha_3$, with $\alpha_j$ the pullback of an element in $H^1(C,\BZ)$ under the $j$-th projection $C^3 \rightarrow C$.  The symmetric group $S_3$ acts naturally on $C^3$, and this induces an action of  $S_3$  on $H^1(C,\BZ)^{\otimes 3}$ with $$\sigma(\alpha_1 \wedge \alpha_2 \wedge \alpha_3) = \mathrm{sgn}(\sigma) \alpha_{\sigma(1)} \wedge \alpha_{\sigma(2)} \wedge \alpha_{\sigma(3)}$$ for each $\sigma \in S_3$. 
Then $(H^1(C,\BZ)^{\otimes 3})^{S_3}  = \bigwedge^3 H^1(C,\BZ)$ for this action, with each member having a basis consisting of $\sum_{\sigma\in S_3} \mathrm{sgn}(\sigma) \alpha_{\sigma(1)} \wedge \alpha_{\sigma(2)} \wedge \alpha_{\sigma(3)}$. 

It is easy to check that the pushforward of $\GS(s)$ to any two factors of $C^3$ is trivial. Hence $\AJ(\GS(s)))  \in \sJ(H^1(C,\BZ)^{\otimes 3}(2))$. Moreover, the modified diagonal is easily seen to be invariant under the action of $S_3$ on $C^3$. Thus $\AJ(\GS(s)) \in  \sJ(\bigwedge^3 H^1(C,\BZ)(2))$.

\begin{lemma}[{\cite[Prop.~2.9]{CG93}}]\label{LemmaGSCe}
$\AJ(\GS(s)) = 3\AJ(\Ce(s))$.
\end{lemma}

Better, we have a morphism of Hodge structures 
$$q_H \colon H^1(C, \BZ) \otimes H^1(C,\BZ) \xrightarrow{\cup} H^2(C,\BZ) \cong \BZ(-1),$$ with $\cup$ the cup product (the dual of the intersection pairing) on $H^1(C,\BZ)$. This defines the {\it contractor}
\[
c\colon \bigwedge\nolimits^3 H^1(C,\BZ)(2) \lra H^1(C,\BZ)(1), \qquad x\wedge y \wedge z \mapsto q_H(y,z)x + q_H(z,x)y + q_H(x,y)z.
\]
Both the source and the target of $c$ are Hodge structures of weight $-1$. So $\ker(c)$ is a Hodge structure of weight $-1$, equipped with a polarization $Q$ induced from $q_H$.
\begin{lemma}\label{LemmaGSNFSubHS}
Both $\AJ(\Delta_{\GS}(s)) $ and $\AJ(\Ce(s))$ are in $\ker(c)$.
\end{lemma}
\begin{proof}
\cite[Lem.~5.1.5]{Zh10} proves $\AJ(\Delta_{\GS}(s)) \in \ker(c)$. Then $\AJ(\Ce(s)) \in \ker(c)$ by Lemma~\ref{LemmaGSCe}.
\end{proof}

Now back to  the universal family $f \colon \sC_g \rightarrow \sM_g$. The fiberwise contractor defines a morphism of variations of Hodge structures
 \begin{equation}\label{EqContractor}
 \mathbf{c}\colon \bigwedge\nolimits^3 R^1f_* \BZ_{\sC_g}(2) \longrightarrow R^1f_* \BZ_{\sC_g}(1).
 \end{equation}
 Let 
\begin{equation}\label{EqDefnVHSGS}
\BV_{\BZ} := \ker(\mathbf{c})
\end{equation}
be its kernel. Then $\BV_{\BZ}$ itself is a variation of Hodge structures, which is of weight $-1$ and carries a polarization $Q$.  Let $V$ be a fiber of the underlying local system $\BV_{\BZ}\otimes_{\BZ_S}\BQ_S$; it is an $\Sp_{2g}$-module and has dimension $\binom{2g}{3} - g$. 

\begin{prop}\label{PropGSNFSubVHS}
The followings are true:
\begin{enumerate}
\item[(i)] both $\nu_{\GS}$ and $\nu_{\Ce}$ have images in $\sJ(\BV_{\BZ})$;
\item[(ii)] neither $\nu_{\GS}$ nor $\nu_{\Ce}$ is a torsion section;
\item[(iii)] the variation of Hodge structures (VHS) $\BV_{\BZ}$ on $\sM_g$ is irreducible, \textit{i.e.} the only sub-VHSs of $\BV_{\BZ}$ are $0$ and itself.
\end{enumerate}
\end{prop}
\begin{proof}
(i) follows immediately from Lemma~\ref{LemmaGSNFSubHS}.

For (ii), it suffices to find a point $s=(C, \xi)$ of genus $g$ such that $\AJ(\Ce(s))$ is not torsion; the result then follows for $\nu_{\GS}$ by Lemma~\ref{LemmaGSCe}. There are many examples of such curves in existing literature. Alternatively, the result for $\nu_{\Ce}$ can be already deduced from Ceresa's original argument in \cite{Ceresa}.

Part (iii), it suffices to prove that $V:=\BV_{\BZ,s}\otimes \BQ$ is a simple $\Sp_{2g}$-module, or equivalently $V$ is an irreducible representation of $\Sp_{2g}$. This is a standard result of the representation theory for $\mathfrak{sp}_{2g}$, see for example \cite[Thm.~17.5]{FultonHarris}.
\end{proof}

\section{Beilinson--Bloch height of  Gross--Schoen and Ceresa cycles}\label{SectionArithmetic}
The goal of this section is to prove Theorem~\ref{ThmZhangBBIntro}. 
More precisely, we first show that the Beilinson--Bloch heights of Gross--Schoen and Ceresa cycles can be expressed as heights of some metrized adelic line bundles on $\sM_g$. We then  relate this metrized line bundle to the metrized line bundles constructed by the normal functions. 


\subsection{Construction of the adelic line bundle}\label{SubsectionConstrALB}
Let $g\ge 3$ and  $\sM_g$ be a fine moduli stack of the pairs $(C, \xi)$ of a smooth projective curve $C$ of genus $g$ and a class $\xi\in \Pic ^1(C)$ 
so that $(2g-2)\xi=\omega_C$, the canconial class of $C$, and let $\sC_g \rightarrow \sM_g$ be the universal curve.

Denote by $\sJ_g = \mathrm{Jac}(\sC_g/\sM_g)$ the relative Jacobian. Identify $\sJ_g$ with its dual via the principal polarization given by a suitable theta divisor.


The Poincar\'{e} line bundle $\CP$  on $\sJ_g \times_{\sM_g} \sJ_g$ extends to an integrable adelic line bundle $\overline{\CP}$ as follows. Define $\CP^{\Delta} := \Delta^*\CP$ 
for the diagonal $\Delta \colon\sJ_g \rightarrow \sJ_g \times_{\sM_g} \sJ_g$. Then $\CP^{\Delta}$ is relatively ample on $\sJ_g \rightarrow \sM_g$, and $[2]^*\CP^{\Delta} = (\CP^{\Delta})^{\otimes 4}$. So $(\sJ_g, [2], \CP^{\Delta})$ is a polarized dynamical system over $\sM_g$ in the sense of \cite[$\S$2.6.1]{YZ21}. Thus, Tate's limit process gives a nef adelic line bundle $\overline{\CP}^{\Delta}$ on $\sJ_g$, as executed by \cite[Thm.~6.1.1]{YZ21}. Now we obtain the desired $\overline{\CP} \in \widehat{\mathrm{Pic}}(\sJ_g/\BZ)_{\BQ}$ by letting
\[
2\overline{\CP} := m^*\overline{\CP}^{\Delta} - p_1^*\overline{\CP}^{\Delta}- p_2^*\overline{\CP}^{\Delta},
\]
where $m, p_1, p_2 \colon \sJ_g \rightarrow \sJ_g \times_{\sM_g} \sJ_g$ with $m$ being the addition and $p_1$ (resp. $p_2$) being the projection to the first (resp. second) factor.

Let 
\[
i_{\mathbf{\xi}} \colon \sC_g \longrightarrow \sJ_g 
\]
be the Abel--Jacobi map based at $\mathbf{\xi}$. Then we have an $\sM_g$-morphism 
$$(i_{\mathbf{\xi}},i_{\mathbf{\xi}}) \colon \sC_g\times_{\sM_g}\sC_g \rightarrow \sJ_g \times_{\sM_g}\sJ_g,$$ and hence get an integrable adelic line bundle on $ \sC_g\times_{\sM_g}\sC_g$
\[
\overline{\CQ} := (i_{\mathbf{\xi}},i_{\mathbf{\xi}})^*\overline{\CP} \in \widehat{\mathrm{Pic}}( \sC_g\times_{\sM_g}\sC_g /\BZ)_{\BQ},
\]
and we can compute
\begin{equation}\label{EqCQ}
\CQ = \CO(\Delta) - p_1^*\xi - p_2^*\xi.
\end{equation}
with $\Delta$ the diagonal of $\sC_g \rightarrow \sC_g \times_{\sM_g}\sC_g$, and $p_1, p_2$ the projections $\sC_g \times_{\sM_g}\sC_g \rightarrow \sC_g$.

Finally the Deligne pairing gives an adelic line bundle on $\sM_g$ by \cite[Thm.~4.1.3]{YZ21}
\begin{equation}
\overline{\CL} := \langle \overline{\CQ}, \overline{\CQ} ,\overline{\CQ} \rangle \in \widehat{\mathrm{Pic}}(\sM_g /\BZ)_{\BQ}.
\end{equation}
The line bundle $\overline{\CL}$ has been constructed \cite[$\S$3.3.4]{YuanArithBig} in slightly different step. 

From our construction above, it is immediate that the Beilinson--Bloch height of the Gross--Schoen cycles is defined by $\overline{\CL}$ by \cite[Thm.~2.3.5]{Zh10}. 
\begin{thm}\label{ThmZhangBB}
For any $s \in \sM_g(\IQbar)$, we have
\[
\langle \GS(s),  \GS(s) \rangle_{\mathrm{BB}} = h_{\overline{\CL}}(s).
\]
\end{thm}

\subsection{Relating with the Gross--Schoen normal function}
Let $\nu_{\GS}$ be the Gross--Schoen normal function $\sM_g \rightarrow \sJ^2(\sC_g^{[3]}/\sM_g)$ defined in \eqref{EqNFGSCe}. 

Set
\[
\overline{\sP^{\mathrm{GS}}} := (\nu_{\GS}, \nu_{\GS})^*\overline{\sP} = \nu_{\GS}^*\overline{\sP^{\Delta}}
\] 
for the metrized Poincar\'{e} bundle $\overline{\sP}$ on $\sJ^2(\sC_g^{[3]}/\sM_g)\times_{\sM_g}\sJ^2(\sC_g^{[3]}/\sM_g)$ defined in Definition~\ref{DefnMetrizedPoincare1} (or the metrized tautological bundle $\overline{\sP^{\Delta}}$ from Definition~\ref{DefnMetrizedPoincare}). Recall the Betti form $
\beta_{\GS}$ constructed from $\nu_{\GS}$ as in Definition~\ref{DefnBettiForm}, 
\textit{i.e.} the curvature of $\overline{\sP^{\mathrm{GS}}}$. It is a semi-positive $(1,1)$-form.
The following result can also be deduced from R. de Jong's work \cite[(9.3)]{dJ16}:

\begin{thm}\label{PropTwoCurvatureEqualAdeGS}
We have the following identity of $(1,1)$-forms on $\sM_g$:
\[
c_1(\overline{\CL}_{\BC}) = \beta_{\GS} .
\]
In particular, $c_1(\overline{\CL}_{\BC})$ is semi-positive.
\end{thm}
\begin{proof}
By the definition of curvatures, we need only construct for each point $s\in \sM_g$, an open neighborhood $U$, and some invertible sections $\ell\in \Gamma (U, \CL)$ and $m\in \Gamma (U, \sP^\GS)$ so that $\|\ell\|=\|m\|$.

First we represent $\xi_U$ by some divisor on $\sC_U:=\sC_g\times _{\sM_g} U$. Then the cycle $\GS(U)$ is represented by a cycle $Z_U$ of $\sC_U^{[3]}$.
Next,  up to shrinking $U$, we can take rational sections $\ell_1, \ell_2, \ell_3$ of $\CQ$ over $\sC_U^2$ such that the product of their divisors is disjoint from 
the self product of $Z_U$ on $\sC_U^6$: 
\[
p_{123}^{-1}|Z_U|\cap p_{456}^{-1}|Z_U|\cap p^{-1}_{14}|t_1|\cap p_{25}^{-1}|t_2|\cap p_{36}^{-1}|t_3|=\emptyset.
\]
Let $W_U$ denote the cycle
\[
W_U=(t_1\otimes t_2\otimes t_3)^*Z_U=p_{123 *}( p_{456}^{-1}Z_U\cdot  p^*_{14}t_1\cap p_{25}^*t_2\cap p_{36}^*t_3).
\]
Then $W_U$ is linearly equivalent to $Z_U$ but disjoint from $Z_U$. Thus, by Hain \cite{Ha90}, we have a well-defined section $m\in \Gamma (U, \sP^\GS)$
such that 
\[
\pair{Z_U, W_U}_\infty (x)=-\log \|m\| (x), \qquad \forall x\in U.
\]
We are done.
\end{proof}

\section{The volume identity}\label{SectionVolIdentity}
In this section, we prove the volume identity relating the (algebraic) volume and analytic volume for the metrized adelic line bundle $\bar \CL$. 
The definition of the (algebraic) volume is
\[
\widehat{\mathrm{vol}}(\wt{\CL}_{\BC}|_S) := \lim_{m\rightarrow \infty}\frac{(\dim S)!}{m^{\dim S}} \hat{h}^0(S,\wt{\CL}_{\BC}|_S^{\otimes m})
\]
where $\hat{h}^0(S,\wt{\CL}_{\BC}|_S^{\otimes m})$ is the dimension of the space of effective sections of $\wt{\CL}_{\BC}|_S^{\otimes m}$. For a precise definition, we refer to \cite[Defn.~5.1.3 and Thm.~5.2.1]{YZ21}. Analytic volume is the integration of volume forms deduced from the curvature forms.

\begin{thm}\label{PropHSforGS}
Let $S \subseteq \sM_{g,\BC}$ be an irreducible subvariety, 
and let $\wt{\CL}_{\BC}|_S$ denote the base change of $\wt{\CL}_{\BC}$ to $S$. Then we have
\begin{equation}\label{EqVolume}
\widehat{\mathrm{vol}}(\wt{\CL}_{\BC}|_S) 
= \int_{S(\BC)}c_1(\overline{\CL}_{\BC})^{\wedge \dim S} = \int_{S(\BC)}\beta_{\GS}^{\wedge \dim S}.
\end{equation}
\end{thm}
\begin{remark}
If $\wt \CL$ is nef, then the above formula follows from the arithmetic Hilbert--Samuel formula. 
Unfortunately, we do not know how to prove the nefness directly. 
\end{remark}

\subsection{Some terminology for adelic line bundles}
For readers' convenience, we briefly recall some terminologies in the theory of adelic line bundles which will be used in our proof of Theorem~\ref{PropHSforGS}.

Let $X$ be a variety defined over a number field $K$. 
An {\em adelic line bundle $\overline{\CL}$ on $X$} is, roughly speaking, the limit of a sequence of {\em model (Hermitian) line bundles} $(\CX_i,\overline\CL_i)$, where $\CX_i$ is a projective arithmetic variety such that $X$ is Zariski open in $\CX_{i,\BQ}$ and $\overline\CL_i$ is a Hermitian line bundle on $\CX_i$. The precise meaning is given by \cite[Chap.~2]{YZ21}. We say that {\it the sequence $(\CX_i,\overline\CL_i)$ converges to $(X,\overline{L})$}.

The group of adelic line bundles on $X$ is usually denoted by $\widehat{\Pic}(X)$. There exists a natural morphism
\[
\widehat{\Pic}(X) \rightarrow \Pic(X),
\]
and denote by $\CL$ the image of $\overline{\CL}$. Moreover for each $i$, the restriction of $\CL_{i,\BQ}$ on $X$ is a line bundle on $X$, and we have a {\it connection morphism} $\ell_i \colon \CL \iso \CL_{i,\BQ}|_X$. Details are also in \cite[Chap.~2]{YZ21}.

\subsection{Proof of Theorem~\ref{PropHSforGS}}
The second equality is immediate by Proposition~\ref{PropTwoCurvatureEqualAdeGS}. We will prove the first equality. 
We start to work on $\sM_g$. To ease notation, denote by $\sM = \sM_g$ and $\sJ = \sJ_g$.

By \cite[Theorem~5.2.1]{YZ21}, it suffices to construct a sequence of model (Hermitian) line bundles $(\CM_i, \overline\CL_i)$  which converges to $(\sM,\overline\CL)$ such that 
\begin{equation}\label{EqGoal}
\lim _{i\rightarrow \infty} \mathrm{vol} (\CL_{i,\BC}|_{S_i})=\int _{S(\BC)} c_1(\overline \CL_\BC)^{\wedge \dim S},
\end{equation}
where $S_i$ is the Zariski closure of $S\subseteq \sM \subseteq \CM_{i,\BC}$.

We start the construction. First, the proof of \cite[Thm.~6.1.1]{YZ21} yields
\begin{itemize}
\item a sequence of projective integral models $\pi_i \colon \CJ_i \rightarrow \CM$ of $\sJ \rightarrow \sM$, 
\item an ample Hermitian line bundle $\overline\CN$ on $\CM$,
\item a sequence of model line bundles $(\CJ_i, \overline \CP^\Delta_i)$ converging to $(\sJ, \overline \CP^\Delta)$, 
\end{itemize}
 such that $\overline \CP^\Delta_i+4^{-i} \pi_i^*\overline \CN$ is nef. Denote by $\ell_i$ the connection morphism $ \CP^\Delta \xrightarrow{\sim} \CP^{\Delta}_{i,\sJ} := \CP^{\Delta}_i \times_{\CJ_i} \sJ$.

For each $n\in \BN$, we have an action of $\sJ[n]$ on $\sJ$ by translating torsion points:
\[
m_n \colon \sJ[n]\times_S \sJ \lra \sJ, \qquad (t, x)\mapsto x+t.
\]
By  Theorem of the square, the norm line bundle $\RN _p m_n^*\CP^\Delta$ is $(\CP^\Delta)^{\otimes n^{2g}}$, where $p$ is the natural projection $\sJ[n]\times_S \sJ \rightarrow \sJ$. 
Thus, we  obtain an adelic line bundle 
$$\overline \CP^\Delta_{i, n}:=n^{-2g} \RN _p m_n ^* \overline \CP^\Delta_i$$
which is  realized on some model $\CJ_{i, n}$ of $\sJ$ with connection morphism $\ell_{i, n}: \CP^\Delta \xrightarrow{\sim} \CP^\Delta_{i, n, \sJ}$.
Moreover $\div (\ell_{i, n})$ is in fact bounded by $\div (\ell _i)$. Over $\BC$, the curvature form  $c_1(\overline \CP^\Delta_{i, n, \BC})$ is obtained from the curvature form $c_1(\overline \CP^\Delta_{i, \BC})$ by taking over average over $n$-torsion points. It follows that these forms converge to $c_1(\overline \CP^\Delta)$ uniformly 
in any compact subset of $\sJ(\BC)$. These bundles also induce a double sequence of model line bundles $(\CM_i, \overline \CL_{i, n})$ of $(\sM, \overline\CL)$ so that they convergent to $(\sM, \overline \CL)$ as $i \rightarrow \infty$, and that the metric $c_1(\overline\CL_{i, n, \BC})$  convergent uniformly to $c_1(\overline \CL_\BC)$ as $n\rightarrow\infty$.

More precisely, we let $\Omega_i$ be an increasing sequence of relatively compact open subsets of $\sM(\BC)$ with $\sM(\BC)=\cup \Omega_i$ and $\epsilon _i$ be a decreasing sequence of positive numbers convergent to $0$ such that  
\begin{equation}\label{eqLN}
c_1(\overline \CL_\BC)\le \epsilon _i^{-1}c_1(\overline \CN_\BC) \quad \text{ on }\Omega_i.
\end{equation}
Then for each $i$, choose $n_i$ such that  
\begin{equation}\label{eq-MN}
-\epsilon_i ^dc_1(\overline \CN_\BC)\le c_1(\overline \CL_{i, n_i, \BC})-c_1(\overline \CL_\BC)\le \epsilon_i ^d c_1(\overline \CN_\BC)
\end{equation}
as Hermitian forms on the tangent bundle of $\Omega_i$. 

For each $i \ge 1$, set $\overline \CL_i := \overline \CL_{i, n_i}$. We will show that the sequence $(\CM_i,\overline\CL_i)$ is the desired sequence of model line bundles such that \eqref{EqGoal} holds.

Using the reference curvature $c_1(\overline\CN_\BC)$, we may talk about eigenvalues of $c_1(\overline\CL_{\BC})$ and $c_1(\overline\CL_{i, \BC})$. Recall that $S_i$ is the closure of $S\subseteq \sM \subseteq \CM_{i,\BC}$ for each $i$. 

From now on, to ease notation we use $\CL_{i,\BC}$ to denote $\CL_{i,\BC}|_{S_i}$.

Let us apply Demailly's  Morse inequality \cite{De91} to the bundle $\overline\CL_{i,\BC}$ on $S_i(\BC)$.
For each $q\in \BN$, let $S_{i,q}$ denote the subset of $S_i(\BC)$ of points where $c_1(\overline \CL_{i,\BC})$ has $q$ negative eigenvalues and $n-q$ positive eigenvalues. Then by \cite[(1.3), (1.5)]{De91}, we have the following estimates as $k\rightarrow \infty$, with $d := \dim S$
$$h^q(k\CL_{i,\BC})\le \frac {k^d}{d!}\left|\int _{S_{i,q}}c_1(\overline \CL_{i,\BC})^{\wedge d}\right|+o(k^d),$$
$$\sum _q(-1)^qh^q(k\CL _{i,\BC})=\frac {k^d}{d!}\int_{S(\BC)} c_1(\overline \CL_{i,\BC})^{\wedge d}+o(k^d).$$
It follows, combined with the definition of $\mathrm{vol}(\CL_{i,\BC})$, that 
\[
\left| \mathrm{vol}(\CL_{i,\BC}) -\int_{S(\BC)} c_1(\overline\CL_{i,\BC})^{\wedge d}\right |\le \sum _{q>0}\left|\int _{S_{i,q}}c_1(\overline \CL_{i,\BC})^{\wedge d}\right|.
\]
By \cite[Thm.~5.4.4]{YZ21}, 
$$\lim _{i\rightarrow \infty} \int_{S(\BC)} c_1(\overline\CL_{i,\BC})^{\wedge d}= \int_{S(\BC)} c_1(\overline\CL_{\BC})^{\wedge d}.$$
Thus to prove \eqref{EqGoal}, it remains to prove 
\begin{equation}\label{EqGoal2}
\lim _{i\rightarrow \infty} \int _{S_{i,q}}c_1(\overline \CL_{i,\BC})^{\wedge d}=0 \quad \text{ for each }q>0.
\end{equation}

We will prove \eqref{EqGoal2} on $\Omega _i\cap S_{i, q}$ and on its complement $S_{i, q}\setminus \Omega_i$ respectively.

On $\Omega _i\cap S_{i, q}$, by  \eqref{eqLN} and \eqref{eq-MN} we have
$$-\epsilon _i^d c_1(\overline\CN_\BC)\le c_1(\overline\CL_{i,\BC})\le (\epsilon_i ^{-1}+\epsilon_i ^d)c_1(\overline\CN_\BC).$$
Thus on $\Omega_i\cap S_{i,q}$, $c_1(\overline\CL_{i,\BC})$ has all eigenvalue $\le \epsilon_i ^{-1}+\epsilon_i ^d$ 
and one negative eigenvalue with absolute value bounded by $\le \epsilon_i^d $.
It follows that $|c_1(\CL_{i,\BC})^{\wedge d}|$ is bounded by 
$$\epsilon_i ^d(\epsilon_i ^{-1}+\epsilon_i ^d)^{d-1}c_1(\overline \CN_{\BC})^{\wedge d}=\epsilon _i(1+\epsilon_i ^{d+1})^{d-1}c_1(\overline \CN_{\BC})^d.$$
It follows  that 
\begin{equation}\label{Goal2Part1}
\int _{S_{i,q}\cap \Omega _i}|c_1(\overline \CL_i)^{\wedge d}|=O(\epsilon _i).
\end{equation}

On $S_{i,q}\setminus \Omega _i$, using 
$$\CP=\frac 12(m^*\CP^\Delta- p_1^*\CP^\Delta-p_2^*\CP^\Delta),$$
we may write $\overline\CL_i=\overline \CE_i-\overline \CF_i$, where $\overline \CE_i$ and $\overline \CF_i$
are   two sequences of new bundles 
convergent to $\overline\CE$ and $\overline\CF$ with smooth metrics respectively.
Then we have 
 \begin{equation}\label{Goal2Part21}
\int _{S_{i,q}\setminus \Omega _i}|c_1(\overline \CL_{i,\BC})^{\wedge d}|\le \int _{S(\BC)\setminus \Omega _i}c_1(\overline \CE_{i,\BC}+\overline \CF_{i,\BC})^{\wedge d}.
 \end{equation}

Now let $\Omega_i'\subseteq \Omega _i$ be another increasing sequence of relatively compact open subsets so that $\cup \Omega _i'=S(\BC)$.
Then we can construct an increasing sequence of continuous functions $f_i$ so that $f_i \equiv 1$ on $S(\BC)\setminus \Omega _i$ 
and $f_i\equiv 0$ on $\Omega _i'$. Then for any $i\ge j$ we have
\[
 \int _{S(\BC)\setminus \Omega _i}c_1(\overline \CE_{i,\BC}+\overline \CF_{i,\BC})^{\wedge d}\le \int _{S(\BC)}f_i\cdot c_1(\overline \CE_{i,\BC}+\overline \CF_{i,\BC})^{\wedge d}
 \le \int _{S(\BC)}f_j\cdot c_1 (\overline \CE_{i,\BC}+\overline \CF_{i,\BC})^{\wedge d}.
\]
 Fix $j$ and take $i\lra \infty$, we get 
 $$\limsup _{i\rightarrow \infty}
 \int _{S(\BC)\setminus \Omega _i}c_1(\overline \CE_{i,\BC}+\overline \CF_{i,\BC})^{\wedge d}\le   \int _{S(\BC)}f_j\cdot c_1 (\overline \CE+\overline \CF)^{\wedge d}\le  
 \int _{S(\BC)\setminus \Omega _j'}c_1 (\overline \CE+\overline \CF)^{\wedge d}.$$
 Letting $j\rightarrow \infty$, we get 
 \begin{equation}\label{Goal2Part2}
 \lim _{i\rightarrow\infty}
 \int _{S(\BC)\setminus \Omega _i} c_1 (\overline \CE_{i,\BC}+\overline \CF_{i,\BC})^{\wedge d}=0.
 \end{equation}
 Now \eqref{EqGoal2} follows immediately from \eqref{Goal2Part1}, \eqref{Goal2Part21} and \eqref{Goal2Part2}. We are done.
\qed

\section{Period map associated with the Gross--Schoen normal function}\label{SectionPeriodMap}

Recall the polarized variation of Hodge structure $\BV_{\BZ}$ on $\sM_g$ of weight $-1$ defined in \eqref{EqDefnVHSGS} which carries a polarization $Q$; in particular  the fiber $\BV_{\BZ,u(\wt{s})}$ is a $\BZ$-Hodge structure of weight $-1$ polarized by $Q$ for each $s \in S(\BC)$. 
Recall the Gross--Schoen normal function $\nu_{\GS} \colon \sM_g \rightarrow \sJ(\BV_{\BZ})$ sending each $s \in \sM_g(\BC)$ to $\AJ(\GS(s))$ (see \eqref{EqNFGSCe} and Proposition~\ref{PropGSNFSubVHS}.(i)).

Let $S$ be an irreducible subvariety of $\sM_{g,\BC}$. By abuse of notation, we will still use $\BV_{\BZ}$, $\sJ(\BV_{\BZ})$ and $\nu_{\GS}$ to denote their restrictions to $S$.

The goal of this section is to define the period map associated with $\nu_{\GS}$, and to relate the Betti foliation $\CF_{\mathrm{Betti}}$ on $\sJ(\BV_{\BZ})$ to the period map.

\subsection{Period map for $\BV_{\BZ} \rightarrow S$}
For the universal cover $u_S \colon \wt{S} \rightarrow S^{\mathrm{an}}$, there is a canonical trivialization $u_S^*\BV_{\BZ} \cong \wt{S} \times V_{\BZ}$ with $V_{\BZ} = H^0(\wt{S},u_S^*\BV_{\BZ})$. 
Thus we have the following {\it universal pure period map} to a certain parametrizing space $\CM_0$ of Hodge structures on 
\[
V := V_{\BZ} \otimes \BQ
\]
 of weight $-1$ polarized by $Q$ (with fixed Hodge numbers)
\[
\wt{\varphi}_0 \colon \wt{S} \rightarrow \CM_0, 
\]
sending  $\wt{s} \in \wt{S}$ to the Hodge structure on $V$ given by the canonical $V = \BV_{\BZ,u(\wt{s})} \otimes\BQ$.

The polarization $Q$ can be seen as a non-degenerate alternating pairing on $V$. It is known that $\CM_0$ is a $\mathbf{G}_0^{\CM}(\BR)^+$-orbit for the  reductive $\BQ$-group $\mathbf{G}_0^{\CM} := \mathrm{Aut}(V,Q)$ -- hence has a real semi-algebraic structure -- and has a natural complex structure.

Finally, since $S$ is a quasi-projective variety, there exists an arithmetic subgroup $\Gamma_0$ of $\mathbf{G}_0^{\CM}(\BQ)$ such that $\wt{\varphi}_0$ descends to a holomorphic map $\varphi_0 \colon S^{\mathrm{an}} \rightarrow \Gamma_0\backslash\CM_0$. For our case of the Gross--Schoen cycles, $\varphi_0$ is known to be injective.

\subsection{Period map for $\nu_{\GS}$}\label{SubsectionPeriodMap}
By \cite[$\S$4.1]{Ha13} and \cite{EZ}, $\nu_{\GS}$ defines a variation of mixed Hodge structures $\BE$ of weight $-1$ and $0$ on $S$, which fits into a short exact sequence $0 \rightarrow \BV_{\BZ} \rightarrow \BE \rightarrow \BZ_S \rightarrow 0$. This short exact sequence, after $\otimes\BQ$ and pullback by $u_S$, is split as a sequence of local systems. Fix such a splitting. 
There is canonical trivialization $u_S^*\BE \cong \wt{S} \times E_{\nu}$ with $E_{\nu} = H^0(\wt{S}, u_S^*\BE)$, and the fixed splitting induces $E_{\nu,\BQ} \cong V\oplus \BQ$. 

As above, we have the {\it universal period map} 
\begin{equation}\label{EqUniversalPeriodMap}
\wt{\varphi} = \wt{\varphi}_{\nu} \colon \wt{S} \longrightarrow \CM 
\end{equation}
where $\CM$ is a parametrizing space  of mixed Hodge structures on $V\oplus \BQ \cong E_{\nu,\BQ}$ of weight $-1$ and $0$ (with fixed Hodge numbers and whose grade $-1$ part is polarized by $Q$). Indeed for each $\tilde{s} \in \wt{S}$, the fiber $\BE_{u(\wt{s})} \otimes \BQ$ is a $\BQ$-mixed Hodge structure of weight $-1$ and $0$, and $\wt{\varphi}$ send $\tilde{s}$ to the Hodge structure on $E_{\nu,\BQ}$ given by the canonical $E_{\nu,\BQ} = \BE_{u(\wt{s})} \otimes \BQ$.

The following subgroup of $\mathrm{GL}(V\oplus \BQ)$
\[
\mathbf{G}^{\CM}:= \left\{\begin{bmatrix} g & v \\ 0 & 1\end{bmatrix} : g \in \mathrm{Aut}(V,Q),~ v \in V\right\} = V \rtimes \mathrm{Aut}(V,Q) = V \rtimes \mathbf{G}_0^{\CM}.
\]
 preserves the filtration $V \subseteq V \oplus \BQ$. By \cite[last Remark of $\S$3]{Pearlstein2000},  $\mathbf{G}^{\CM}(\BR)^+$ acts transitively on $\CM$ by Lemma~\ref{LemmaMTGroup} and its proof. So $\CM = \mathbf{G}^{\CM}(\BR)^+h$ for any $h \in \CM$.

Again, there exists an arithmetic subgroup $\Gamma$ of $\mathbf{G}^{\CM}(\BQ)$ such that $\wt{\varphi}$ descends to a holomorphic map $\varphi \colon S^{\mathrm{an}} \rightarrow \Gamma\backslash\CM$ since $S$ is a quasi-projective variety. Then $\varphi$ is injective because $\varphi_0$ is.

We close this subsection with the definition of the group $\mathrm{MT}(\nu_{\GS}(S))$. By \cite[$\S$4, Lem.~4]{Andre92}, the Mumford--Tate group $\mathrm{MT}(\BE_s\otimes \BQ)$ of the $\BQ$-mixed Hodge structure $\BE_s\otimes \BQ$ is locally constant for a very general point $s \in S(\BC)$.  This group is denoted by $\mathrm{MT}(\nu_{\GS}(S))$ and is known to be a subgroup of $\mathbf{G}^{\CM}$.

\subsection{Relating the two period maps and the Betti foliation}
The quotient $p \colon \mathbf{G}^{\CM} \rightarrow \mathbf{G}_0^{\CM} = \mathbf{G}^{\CM}/V$ induces a surjective map $\CM  \rightarrow \CM_0$, which by abuse of notation we still denote by $p$. Then we have the following commutative diagram
\begin{equation}\label{EqPeriod}
\xymatrix{
& \CM \ar[d]^-{p} \ar[r]^-{u} & \Gamma \backslash \CM \ar[d]^-{[p]} \\
\wt{S} \ar[r]_-{\wt{\varphi}_0} \ar[ru]^-{\wt{\varphi}} & \CM_0 \ar[r]^-{u_0} & \Gamma_0 \backslash \CM_0
}
\end{equation}
where $u$ and $u_0$ are the uniformizations. 

Moreover since $\mathbf{G}^{\CM} = V \rtimes \mathbf{G}_0^{\CM}$, we have 
\begin{equation}\label{EqSemiAlgebraicStructureCM}
\CM \xrightarrow{\sim} V(\BR) \times \CM_0.
\end{equation}
The semi-algebraic structure on $\CM$ is given by this product. The complex structure on $\CM$ is given as follows. Any $h_0 \in \CM_0$ parametrizes a Hodge structure on $V$ of weight $-1$, whose Hodge filtration we denote by $\sF^{\bullet}_{h_0}V_{\BC}$. The complex structure on the fiber $p^{-1}(h_0)$ is  given by $V_{\BR} \cong V_{\BC}/\sF^{\bullet}_{h_0} V_{\BC}$. This makes $\CM$ into a complex vector bundle on $\CM_0$ such that, for any $a \in V(\BR)$, the subset $\{a\}\times \CM_0 \subseteq V(\BR)\times \CM_0 = \CM$ is complex analytic.

\medskip

Now let us explain how the intermediate Jacobian $\pi \colon \sJ(\BV_{\BZ}) \rightarrow S$ and the Betti foliation $\CF_{\mathrm{Betti}}$, both defined in $\mathsection$\ref{SubsectionBettiFoliation}, are related to this diagram \eqref{EqPeriod}. Because $V_{\BZ} = H^0(\wt{S},u_S^*\BV_{\BZ})$ is a lattice in $V_{\BR}$ and $u_S^*\BV_{\BZ} \cong V_{\BZ} \times \wt{S}$, we have $\sJ(\BV_{\BZ}) \times_S \wt{S} \cong (\BV_{\BR}/\BV_{\BZ}) \times_S \wt{S} = V_{\BR}/V_{\BZ} \times \wt{S}$. For the universal covering map $u_{\sJ} \colon \wt{\sJ} \rightarrow \sJ(\BV_{\BZ})$, this furthermore induces 
\begin{equation}\label{EqUniversalSJ}
\wt{\sJ} \xrightarrow{\sim} V(\BR) \times \wt{S}.
\end{equation}
 So the leaves of the Betti foliation $\CF_{\mathrm{Betti}}$ are precisely those $u_{\sJ}(\{a\}\times \wt{S})$ for all $a \in V(\BR)$.

It can be related to periods maps as follows.  Each $x \in \sJ(\BV_{\BZ})$ lies in $\sJ(\BV_{\BZ})_s = \sJ(\BV_{\BZ,s})$ for $s =\pi(x)$, 
and hence gives rise to a $\BZ$-mixed Hodge structure which is an extension of $\BV_{\BZ,s}$ by $\BZ(0)$ 
 by Carlson \cite{Carlson85}.  So each $\wt{x} \in \wt{\sJ}$ gives rise to a $\BQ$-mixed Hodge structure on $V\oplus \BQ$ of weight $-1$ and $0$. 
So we obtain a period map $\wt{\varphi}_{\sJ} \colon \wt{\sJ} \rightarrow \CM$, which by the constructions is the composite
\begin{equation}\label{EqPeriodJacobian}
\wt{\varphi}_{\sJ} \colon \wt{\sJ} \xrightarrow{\sim} V(\BR) \times \wt{S} \xrightarrow{(1,  \wt{\varphi}_0)} V(\BR) \times \CM_0 \xrightarrow{\sim} \CM,
\end{equation}
where the first isomorphism is \eqref{EqUniversalSJ} and the last isomorphism is   \eqref{EqSemiAlgebraicStructureCM}. For our case of Gross--Schoen cycles, each fiber of $\wt{\varphi}_{\sJ}$ has dimension $0$ because $\varphi_0$ is injective.
 
To ease notation, we identify $\CM = V(\BR) \times \CM_0$ using \eqref{EqSemiAlgebraicStructureCM}. Then the leaves of $\CF_{\mathrm{Betti}}$ are precisely $u_{\sJ}\left( \wt{\varphi}_{\sJ}^{-1}(\{a\}\times \CM_0) \right)$ for $a \in V(\BR)$.

Now we go back to $\wt{\varphi} = \wt{\varphi}_{\nu} \colon \wt{S} \rightarrow \CM$ from \eqref{EqUniversalPeriodMap}. 
The map $\nu_{\GS} \circ u_S \colon \wt{S} \rightarrow S \rightarrow \sJ(\BV_{\BZ})$ lifts to $\wt{\nu} \colon \wt{S} \rightarrow \wt{\sJ}$, and $\wt{\varphi} = \wt{\varphi}_{\sJ}\circ \wt{\nu}$.

The Betti stratum defined in \eqref{EqDefnBettiStratum} is 
$$S^{\mathrm{Betti}}(1) = \{s \in S(\BC): \dim_{\nu_{\mathrm{GS}}(s)} (\nu_{\mathrm{GS}}(S) \cap \CF_{\mathrm{Betti}}) \ge 1 \}$$
 in our case.
\begin{lemma}\label{LemmaBettiFoliationBettiRank}
Under the identification $\CM = V(\BR) \times \CM_0$ given by \eqref{EqSemiAlgebraicStructureCM}, we have
\[
u_S^{-1}\left(S^{\mathrm{Betti}}(1)\right)  =  \bigcup_{\substack{\{a\}\times \wt{C} \subseteq \wt{\varphi}(\wt{S}), \qquad a \in V(\BR) \\ \wt{C}\text{ is a complex analytic curve in }\CM_0}} (\{a\}\times \wt{C}).
\]
\end{lemma} 
\begin{proof}
Consider $\nu_{\GS}(S^{\mathrm{Betti}}(1)) \subseteq \sJ(\BV_{\BZ})$. Recall that each fiber of $\wt{\varphi}_{\sJ}$ has dimension $0$. 
By the characteriation of the leaves of $\CF_{\mathrm{Betti}}$ below \eqref{EqPeriodJacobian}, we have  
$$u_{\sJ}^{-1}\left( \nu_{\GS}(S^{\mathrm{Betti}}(1))\right) = \{ \wt{s}' \in \wt{\nu}(\wt{S}) : \dim_{\wt{s}'} \wt{\varphi}_{\sJ}^{-1}(\{\wt{s}'_V\} \times \CM_0\}) \ge  1\},$$
 where $\wt{s}'_V$ is the image of 
 $$\wt{s} \in \wt{\sJ} \xrightarrow{\wt{\varphi}_{\sJ}} V(\BR) \times \CM_0 \rightarrow V(\BR)$$ with the last map being the natural projection. 
Applying $\wt{\nu}^{-1}$ (whose fibers are of dimension $0$ because $\nu_{\GS}$ is injective) to the set above and noticing that $u_{\sJ}\circ \wt{\nu}=\nu_{\GS}\circ u_S$, we have
\[
u_S^{-1}\left(S^{\mathrm{Betti}}(1)\right) = \{\wt{s} \in \wt{S} : \dim_{\wt{s}} (\wt{\varphi}_{\sJ}\circ \wt{\nu})^{-1}(\{\wt{s}'_V\} \times \CM_0\}) \ge 1\}.
\] 
So we can conclude because $ \wt{\varphi}_{\sJ}\circ \wt{\nu}=\wt{\varphi} $.
\end{proof}

\subsection{Mumford--Tate domains}
In practice, we often need to study subsets of $\CM$ parametrizing Hodge structures with extra data, namely the 
 {\it Mumford--Tate domains}. Let us briefly recall the definition. 
Each point $x \in \CM$ parametrizes a $\BQ$-mixed Hodge structure on $V\oplus \BQ$, and let $\mathrm{MT}_x < \mathrm{GL}(V\oplus \BQ)$ be its Mumford--Tate group. 
It is known that $\mathrm{MT}_x$ is a subgroup of $\mathbf{G}^{\CM}$ and has unipotent radical $V \cap \mathrm{MT}_x$.

A {\it Mumford--Tate domain} in $\CM$ is a subset of the form $\mathbf{G}(\BR)^+x \subseteq \CM$ with $x \in \CM$ and $\mathbf{G} = \mathrm{MT}_x$. 
Mumford--Tate domains are easily seen to be semi-algebraic in $\CM$, and are known to be complex analytic in $\CM$.

Recall the group quotient  $p \colon \mathbf{G}^{\CM} = V \rtimes  \mathbf{G}_0^{\CM} \rightarrow \mathbf{G}_0^{\CM}$ and the induced projection $p \colon \CM = V(\BR) \times \CM_0 \rightarrow \CM_0$. 

Let $\sD = \mathbf{G}(\BR)^+x$ be a Mumford--Tate domain in $\CM$. Denote by $V_{\sD} := V \cap \mathbf{G}$ the unipotent radical of $\mathbf{G}$ and by $\sD_0 := p(\sD) \subseteq \CM_0$.
\begin{lemma}\label{LemmaMTDomain}
Under the identification $\CM = V(\BR) \times \CM_0$, we have $\sD = (v_0 + V_{\sD}(\BR)) \times \sD_0$ for some $v_0 \in V(\BQ)$.
\end{lemma}
\begin{proof} Denote by $\mathbf{G}_0 := \mathbf{G}/V_{\sD}$ the reductive part of $\mathbf{G}$, which can be seen as a subgroup of $\mathbf{G}_0^{\CM} = \mathbf{G}^{\CM}/V$. 
Now $\mathbf{G}_0$ can be seen as a subgroup of $\mathbf{G}^{\CM}$ via $\mathbf{G}_0 < \mathbf{G}_0^{\CM} = \{0\}\times \mathbf{G}_0^{\CM} < V \rtimes  \mathbf{G}_0^{\CM} = \mathbf{G}^{\CM}$. Thus we obtain a subgroup $\mathbf{G}':=V_{\sD} \rtimes \mathbf{G}_0$ of $\mathbf{G}^{\CM}$. The orbit $\mathbf{G}'(\BR)^+x$ is $V_{\sD}(\BR) \times \sD_0$ in $V(\BR) \times \CM_0 = \CM$.

Now compare $\mathbf{G}$ and $\mathbf{G}'$. Take a Levi decomposition $\mathbf{G} = V_{\sD} \rtimes' \mathbf{G}_0$, then there exists a compatible Levi decomposition of $\mathbf{G}^{\CM} = V \rtimes' \mathbf{G}_0^{\CM}$ as in \cite[proof of Prop.~2.9]{GaoTowards-the-And}.  General group theory says that this new Levi decomposition differs from the original one $\mathbf{G}^{\CM} = V \rtimes  \mathbf{G}_0^{\CM} $ by the conjugation by some $v_0 \in V(\BQ)$. So $\mathbf{G} = v_0 \mathbf{G}' v_0^{-1}$. So $\mathbf{G}(\BR)^+ x = (v_0+V_{\sD}(\BR)) \times \sD_0$ by the last sentence of the paragraph above. 
\end{proof}

\subsection{Quotient by a normal subgroup and weak Mumford--Tate domains}
Let $\sD = \mathbf{G}(\BR)^+x$ be a Mumford--Tate domain and let $N \lhd \mathbf{G}$ be a normal subgroup. By \cite[Prop.~5.1]{GKAS}, we have a quotient in the category of complex varieties
\[
p_N \colon \sD \rightarrow \sD/N
\]
with each fiber being an $N(\BR)^+$-orbit. 
For the arithmetic subgroup $\Gamma$ of $\mathbf{G}^{\CM}(\BQ)$ in $\mathsection$\ref{SubsectionPeriodMap}, let $\Gamma_{\mathbf{G}} := \Gamma \cap \mathbf{G}(\BQ)$ and $\Gamma_{\mathbf{G}/N} : =\Gamma_{\mathbf{G}}/(\Gamma\cap N(\BQ))< (\mathbf{G}/N)(\BQ)$. Then $p_N$ induces
\begin{equation}\label{EqpNQuotientAfterModGamma}
[p_N] \colon \Gamma_{\mathbf{G}}\backslash \sD \rightarrow \Gamma_{\mathbf{G}/N} \backslash (\sD/N).
\end{equation}

We will pay special attention to the fibers of $p_N$, for which we define:
\begin{defn}\label{DefnWSP}
A subset $\sD_N$ of $\sD$ is called a {\it weak Mumford--Tate domain} if there exist $x \in \sD$ and a normal subgroup $N$ of $\mathrm{MT}_x$ such that $\sD_N = N(\BR)^+x$.
\end{defn}
Weak Mumford--Tate domains are easily seen to be semi-algebraic in $\CM$, and are known to be also complex analytic in $\CM$. 

\begin{lemma}\label{LemmaWeakMTDomain}
For a weak Mumford--Tate domain $\sD_N$ as in Definition~\ref{DefnWSP}, under the identification $\CM = V(\BR) \times \CM_0$ we have $\sD_N = (v + V_N(\BR)) \times p(\sD_N)$ for some $v\in V(\BR)$, where $V_N = V \cap N$.
\end{lemma}
\begin{proof}
By Lemma~\ref{LemmaMTDomain}, we have $\sD = (v_0 + V_{\sD}(\BR)) \times \sD_0$, with $V_{\sD} = V \cap \mathbf{G}$ and $v_0 \in V(\BQ)$. Write $\mathbf{G}_0 := \mathbf{G}/V_{\sD}$, and $N_0 := N/V_N \lhd \mathbf{G}_0$. Then $N_0$ acts on $V$ via $N_0 \lhd \mathbf{G}$. Notice that $p(\sD_N) \subseteq \sD_0$. Write $x = (v_0 + v_x, x_0) \in V(\BR) \times \CM_0$, with $v_x \in V_{\sD}(\BR)$.

Since $N \lhd \mathbf{G}$, we have: (i) $V_N$ is a $\mathbf{G}$-module; (ii) the induced action of $N_0 \lhd \mathbf{G}_0$ on $V_{\sD}/V_N$ is trivial. Thus $g \cdot v_x + V_N(\BR) = v_x + V_N(\BR)$ for all $g \in N_0(\BR)^+$. So we can conclude with $v = v_0 + v_x$.
\end{proof}

\section{Non-vanishing of the Betti form}\label{SectionBettiRank}
In this section we prove that the top wedge power of the Betti form is not identically zero on $\sM_g$, and that the Betti stratum is Zariski closed.
\begin{thm}\label{ThmBettiRankFormulaMain}
We have $\beta_{\mathrm{GS}}^{\wedge \dim \sM_g} \not\equiv 0$ for $g \ge 3$.
\end{thm}

\begin{thm}\label{ThmZariskiClosedDegLoci}
For any subvariety $S\subseteq \sM_g$, the subset 
$$S^{\mathrm{Betti}}(1) := \{s \in S(\BC): \dim_{\nu_{\mathrm{GS}}(s)} (\nu_{\mathrm{GS}}(S) \cap \CF_{\mathrm{Betti}}) \ge 1 \}$$ 
is Zariski closed in $S$. 
\end{thm}

The proofs of Theorem~\ref{ThmBettiRankFormulaMain} and Theorem~\ref{ThmZariskiClosedDegLoci} are simultaneous and follow the guideline of the first-named author's \cite{GaoBettiRank} on the case when $\sJ(\BV_{\BZ})$ is polarizable. In fact, without much effort our proof can be generalized to show the following criterion: For any subvariety $S \subseteq \sM_g$, we have 
\small
\begin{equation}\label{EqCritBettiRank}
\beta_{\GS}|_S^{\wedge \dim S} \not\equiv 0 \Longleftrightarrow \dim S \le \dim [p_N](S) + \frac{1}{2}\dim_{\BQ}(V\cap N) \text{ for any }N \lhd \mathrm{MT}(\nu_{\GS}(S)),
\end{equation}
\normalsize
where $\mathrm{MT}(\nu_{\GS}(S))$ is defined at the end of $\S$\ref{SubsectionPeriodMap}. 
We will not prove \eqref{EqCritBettiRank} in this section but refer to Theorem~\ref{ThmBettiRankGeneralFamily} for the more general situation.


For the proof, we will use the notation from \eqref{EqPeriod}. In particular, the uniformization
\[
u \colon \CM \rightarrow \Gamma\backslash \CM 
\]
and the period map $\varphi \colon S \rightarrow \Gamma\backslash \CM$ which is injective. By abuse of notation we use $S$ to denote $\varphi(S)$. Then $S$ is a complex analytic subvariety of $\Gamma\backslash\CM$. By the main result of \cite{BBKT}, we can endow $\Gamma\backslash\CM$ with the following natural o-minimal definable structure on  such that $S$ is a definable subvariety: let $\mathfrak{F} \subseteq \CM$ be a real semi-algebraic fundamental domain for the action of $\Gamma$ on $\CM$, then as sets $\Gamma\backslash\CM$ can be identified with $\mathfrak{F}$, and this identification gives $\Gamma\backslash\CM$ a real semi-algebraic structure, which is definable in $\BR_{\mathrm{an},\exp}$.

\subsection{Bi-algebraic system and Ax--Schanuel}
\begin{prop}[{\cite[Cor.~6.7]{BBKT}}]\label{PropWeakMTBialg}
Let $\sD_N$ be a weak Mumford--Tate domain defined in Definition~\ref{DefnWSP}. Then $u(\sD_N) \cap S$ is a closed algebraic subset of $S$. 
\end{prop}
This proposition follows immediately from definable Chow because $u(\sD_N) \cap S$ is naturally complex analytic and is definable by the definition of the o-minimal definable structure on $\Gamma\backslash\CM$ given by \cite{BBKT}.

\begin{defn}\label{DefnWeaklySpecialClosure}
\begin{enumerate}
\item[(i)] Let $\wt{Y} \subseteq \CM$ be a complex analytic irreducible subset. The \textbf{weakly special closure} of $\wt{Y}$, denoted by $\wt{Y}^{\mathrm{ws}}$, is the smallest weak Mumford--Tate domain in $\CM$ which contains $\wt{Y}$.
\item[(ii)] Let $Y \subseteq S$ be an irreducible subvariety. The \textbf{weakly special closure} of $Y$, denoted by  $ Y^{\mathrm{ws}}$, is $u(\widetilde{Y}^{\mathrm{ws}})$ for one (hence any) complex analytic irreducible component $\widetilde{Y}$ of $u^{-1}(Y)$.
\end{enumerate}
\end{defn}

We also need to use the algebraic structure on $\CM$ (resp. on $\CM_0$), usually defined by the natural inclusion of $\CM$ into a complex algebraic variety $\CM^\vee$ (resp. into $\CM_0^\vee$). For our purpose, it is more convenient to use the following equivalent definition (the equivalence follows from \cite[Lem.~B.1 and proof]{KlinglerThe-Hyperbolic-}).
\begin{defn}
A subset of $\CM$ (resp. of $\CM_0$) is said to be {\it irreducible algebraic} if it is both complex analytic irreducible and semi-algebraic.
\end{defn}

The following Ax--Schanuel theorem for VMHS was independently proved by Chiu \cite{KennethChiuAS} and Gao--Klingler \cite{GKAS}.

\begin{thm}[weak Ax--Schanuel for VMHS]\label{ThmAS}
Let $\wt{Z} \subseteq u^{-1}(S)$ be a complex analytic irreducible subset, with $\wt{Z}^{\mathrm{Zar}}$ the smallest irreducible algebraic subset containing $\wt{Z}$. Then
\begin{equation}\label{EqWAS}
\dim \wt{Z}^{\mathrm{Zar}} + \dim u(\wt{Z})^{\mathrm{Zar}}  \ge \dim  \wt{Z}^{\mathrm{ws}} + \dim \wt{Z},
\end{equation}
where $\wt{Z}^{\mathrm{ws}}$ is the smallest weak Mumford--Tate domain which contains $\wt{Z}$.
\end{thm}
\begin{proof}
Let $Y := u(\wt{Z})^{\mathrm{Zar}}$. 
Let 
$$\mathscr{Z} := \{(z,y) \in \wt{Z} \times Y(\BC) : u(z) = y\},$$ then $\mathscr{Z}$ is a complex analytic irreducible subset of $\CM \times_{\Gamma\backslash\CM} Y'$. The Zariski closure of $\mathscr{Z}$ in $\CM \times Y$ is contained in $\wt{Z}^{\mathrm{Zar}} \times Y$, 
 and $\dim \mathscr{Z} = \dim \wt{Z}$. Then \eqref{EqWAS} is a direct consequence of the mixed Ax--Schanuel theorem \cite[Thm.~1.1]{GKAS} applied to $\mathscr{Z}$. 
\end{proof}

We close this introductory subsection with the following definition. In practice, we often need to work with algebraic subvarieties $Y\subseteq S$, which are not weak Mumford--Tate domains, and the following number measures how far it is from being one.
\begin{equation}
\delta_{\mathrm{ws}}(Y) := \dim  Y^{\mathrm{ws}} - \dim Y.
\end{equation}

\begin{defn}
An irreducible algebraic subvariety $Y$ of $S$ is called {\em weakly optimal} if the following holds: 
$$Y \subsetneq Y' \subseteq S \Rightarrow \delta_{\mathrm{ws}}(Y) < \delta_{\mathrm{ws}}(Y'),$$ for any $Y' \subseteq S$ irreducible.
\end{defn}

\subsection{Applications of Ax--Schanuel}
The following proposition is an application of mixed Ax--Schanuel.
\begin{prop}\label{PropBettiRankNonMaxAndDegLocus}
Recall the projection $[p] \colon \Gamma\backslash\CM \rightarrow \Gamma_0 \backslash \CM_0$ from \eqref{EqPeriod}. Then 
$S^{\mathrm{Betti}}(1)$
is contained in the union of weakly optimal subvarieties $Y \subseteq S$ satisfying
\begin{equation}\label{EqDegenerateMemberUnion}
\dim Y \ge \dim Y^{\mathrm{ws}} -  \dim [p](Y^{\mathrm{ws}}) +1.
\end{equation}
\end{prop}
\begin{proof}
It suffices to prove two things:
\begin{enumerate}
\item[(i)] $S^{\mathrm{Betti}}(1)$ is covered by the union of irreducible subvarieties $Y \subseteq S$ satisfying \eqref{EqDegenerateMemberUnion} (without requiring $Y$ to be weakly optimal);
\item[(ii)] If $Y \subseteq S$ is an irreducible subvariety satisfying \eqref{EqDegenerateMemberUnion} and is maximal for this property with respect to inclusions, then $Y$ is weakly optimal.
\end{enumerate}

Let us prove (i). By Lemma~\ref{LemmaBettiFoliationBettiRank}, $S^{\mathrm{Betti}}(1) $ is covered by irreducible subvarieties $Y \subseteq S$ such that
\[
Y:= \overline{u(\{a\} \times \wt{C})}^{\mathrm{Zar}},
\]
 for some complex analytic curve $\wt{C} \subseteq \CM_0$  and some $a \in V(\BR)$.

Apply mixed Ax--Schanuel in this context (Theorem~\ref{ThmAS} to $\{a\} \times \wt{C}$). Then we get
\begin{equation}\label{EqAppAS}
\dim (\{a\} \times \wt{C})^{\mathrm{Zar}} + \dim Y\ge \dim ( \{a\} \times \wt{C} )^{\mathrm{ws}} + 1.
\end{equation}
The following is easy to check by the discussion below \eqref{EqSemiAlgebraicStructureCM}: For the smallest irreducible algebraic subset $\wt{C}^{\mathrm{Zar}}$ of $\CM_0$ which contains $\wt{C}$, the subset 
$$\{a\} \times \wt{C}^{\mathrm{Zar}} \subseteq V(\BR) \times \CM_0 = \CM$$ is both semi-algebraic and complex analytic in $\CM$. Therefore $(\{a\} \times \wt{C})^{\mathrm{Zar}}  = \{a\} \times \wt{C}^{\mathrm{Zar}}$, and thus $\dim (\{a\} \times \wt{C})^{\mathrm{Zar}} =  \dim \wt{C}^{\mathrm{Zar}}$. So \eqref{EqAppAS} implies
\[
\dim Y  \ge \dim ( \{a\} \times \wt{C})^{\mathrm{ws}} -  \dim \wt{C}^{\mathrm{Zar}} + 1.
\]
So, to prove (i), it suffices to prove $\dim   ( \{a\} \times \wt{C} )^{\mathrm{ws}}  = \dim   Y^{\mathrm{ws}}$ and $\dim  \wt{C}^{\mathrm{Zar}} \le \dim [p](  Y^{\mathrm{ws}})$.

Let us prove $u( ( \{a\} \times \wt{C} )^{\mathrm{ws}} ) =  Y^{\mathrm{ws}}$; the upshot is $\dim  ( \{a\} \times \wt{C} )^{\mathrm{ws}}  = \dim   Y^{\mathrm{ws}}$. First of all, we have $u( ( \{a\} \times \wt{C} )^{\mathrm{ws}} ) = u( \{a\} \times \wt{C} )^{\mathrm{ws}} $ by Definition~\ref{DefnWeaklySpecialClosure}. By definition of $Y$, we have $u(\{a\} \times \wt{C}) \subseteq Y$. Hence $u(  \{a\} \times \wt{C} )^{\mathrm{ws}}   \subseteq  Y^{\mathrm{ws}}$. On the other hand, $u( ( \{a\} \times \wt{C} )^{\mathrm{ws}})  \cap S$ is closed algebraic by Proposition~\ref{PropWeakMTBialg}, so 
$$Y \subseteq u( ( \{a\} \times \wt{C} )^{\mathrm{ws}}  ) = u( \{a\} \times \wt{C} )^{\mathrm{ws}}.$$ So $ Y^{\mathrm{ws}} \subseteq u( \{a\} \times \wt{C} )^{\mathrm{ws}}$. Now we have established 
$$ Y^{\mathrm{ws}} = u( \{a\} \times \wt{C} )^{\mathrm{ws}} = u( ( \{a\} \times \wt{C} )^{\mathrm{ws}} ).$$

Let us prove $\dim  \wt{C}^{\mathrm{Zar}} \le \dim [p](  Y^{\mathrm{ws}})$. We have 
$$(\{a\}\times\wt{C})^{\mathrm{Zar}} \subseteq (\{a\}\times\wt{C})^{\mathrm{ws}}.$$
So $\{a\} \times \wt{C}^{\mathrm{Zar}} \subseteq (\{a\}\times\wt{C})^{\mathrm{ws}}$. Apply $p \colon \CM \rightarrow \CM_0$ to both sides, we get $$\wt{C}^{\mathrm{Zar}} \subseteq p((\{a\}\times\wt{C})^{\mathrm{ws}}).$$
So $$\dim  \wt{C}^{\mathrm{Zar}} \le \dim p((\{a\}\times\wt{C})^{\mathrm{ws}}) = \dim (u_0\circ p)((\{a\}\times\wt{C})^{\mathrm{ws}})$$ for the uniformization $u_0 \colon \CM_0 \rightarrow \Gamma_0 \backslash\CM_0$. But $u_0\circ p = [p]\circ u$ by \eqref{EqPeriod}. So 
$$\dim  \wt{C}^{\mathrm{Zar}} \le[p]\circ u( ( \{a\} \times \wt{C} )^{\mathrm{ws}} ),$$ and the right hand side equals $\dim [p](Y^{\mathrm{ws}})$ by last paragraph.

Hence we are done for (i). 

For (ii), let $Y \subseteq Y' \subseteq X$. Assume $\delta_{\mathrm{ws}}(Y) \ge \delta_{\mathrm{ws}}(Y')$, \textit{i.e.} 
\[
\dim  Y^{\mathrm{ws}} - \dim Y \ge \dim  Y^{\prime,\mathrm{ws}} - \dim Y'.
\]
The assumption on $Y$ implies $$\dim  Y^{\mathrm{ws}} - \dim Y \le \dim [p](   Y^{\mathrm{ws}} ) -1.$$
Combined with the inequality above, we obtain $$\dim  Y^{\prime,\mathrm{ws}} - \dim Y' \le  \dim [p](   Y^{\prime,\mathrm{ws}} ) -1$$ because $Y \subseteq Y'$. Therefore $Y = Y'$ by maximality of $Y$. Hence, (ii) is established.
\end{proof}

To continue, we need the following finiteness proposition for weakly optimal subvarieties in $S$ \textit{\`{a} la Ullmo}, which itself is an application of mixed Ax--Schanuel. 
In recent literature, it is called {\it Geometric Zilber--Pink} and  weakly optimal subvarieties are called {\it monodromically atypical maximal}. When $\sJ(\BV_{\BZ})$ is polarizable, \textit{i.e.} in the Shimura case, it was proved by the first-named author \cite[Thm.~1.4]{GaoMixedAS}. The version we need here is  Baldi--Urbanik \cite[Thm.~7.1]{BaldiUrbanik} {\it in the case of weight $-1$ and $0$} (in this case the factor $N^u(\BC)$ in \cite[Thm.~7.1]{BaldiUrbanik} is not needed because of Lemma~\ref{LemmaMTGroup}). 
\begin{prop}\label{PropFinitenessBogomolov}
There exist finitely many pairs $(\sD_1,N_1), \ldots, (\sD_k,N_k)$, with $\sD_j = \mathbf{G}_j(\BR)^+x_j$ a Mumford--Tate domain in $\CM$ and $N_j \lhd \mathbf{G}_j$, such that the following holds. For each weakly optimal subvariety $Y \subseteq S$,  $ Y^{\mathrm{ws}}$ is a fiber of $[p_{N_j}]$ for some $j \in \{1,\ldots,k\}$.
\end{prop}

Here $[p_{N_j}]$ is the complex analytic map defined as in \eqref{EqpNQuotientAfterModGamma} with $\mathbf{G}$ replaced by $\mathbf{G}_j$ and $N$ replaced by $N_j$. More precisely, 
denote by $\Gamma_j = \Gamma \cap \mathbf{G}_j(\BQ)$ and $\Gamma_{j,/N_j} = \Gamma_j / (\Gamma_j \cap N_j(\BQ))$. Then $[p_{N_j}]$ is the natural projection $u(\sD_j) = \Gamma_j \backslash\sD_j \rightarrow \Gamma_{j,/N_j}\backslash (\sD_j/N_j)$, which is easily seen to be real semi-algebraic (hence definable) for the definable structures given above Proposition~\ref{PropWeakMTBialg}.


\subsection{Proof of Theorem~\ref{ThmBettiRankFormulaMain} and Theorem~\ref{ThmZariskiClosedDegLoci}}\label{SubsectionProofGeom}
For each $j \in \{1,\ldots,k\}$, Proposition~\ref{PropWeakMTBialg} says that $u(\sD_j)\cap S$ is a closed algebraic subset of $S$.  The restriction 
\[
[p_{N_j}]|_S \colon u(\sD_j) \cap S \rightarrow \Gamma_{j,/N_j}\backslash (\sD_j/N_j)
\]
is both complex analytic and  definable because $[p_{N_j}]$ is. 
Hence the subset
\begin{equation}\label{EqEjSet}
E_j := \left\{ s\in u(\sD_j) \cap S : \dim_s [p_{N_j}]|_S^{-1}([p_{N_j}](s)) \ge \frac{1}{2}\dim_{\BQ} (V \cap N_j) +1 \right\}
\end{equation}
is both definable and complex analytic in $u(\sD_j)\cap S$. So $E_j$ is algebraic by definable Chow. Moreover, it is closed in $u(\sD_j)\cap S$ by the upper semi-continuity of fiber dimensions. So $E_j$ is a closed algebraic subset of $S$.

\begin{lemma}\label{PropFoliationDegLociSubset}
$S^{\mathrm{Betti}}(1) \subseteq \bigcup\nolimits_{j=1}^k E_j$.
\end{lemma}
\begin{proof}
By Proposition~\ref{PropBettiRankNonMaxAndDegLocus}, $S^{\mathrm{Betti}}(1)$ 
 is covered by weakly optimal $Y \subseteq S$ such that 
 $\dim Y \ge  \dim  Y^{\mathrm{ws}} -  \dim [p]( Y^{\mathrm{ws}}) +1$.  By Proposition~\ref{PropFinitenessBogomolov}, $ Y^{\mathrm{ws}}$ is a fiber $[p_{N_j}]$ for some $j \in \{1,\ldots,k\}$. Each fiber of $p \colon \CM \rightarrow \CM_0$ is a $V(\BR)$-orbit and each fiber of $p_{N_j}$ is an $N_j(\BR)^+$-orbit. So  $$\dim_{\BR}  Y^{\mathrm{ws}} - \dim_{\BR} [p]( Y^{\mathrm{ws}}) = \dim_{\BQ}  (V \cap N_j).$$ So $\dim Y \ge  \frac{1}{2}\dim_{\BQ}  (V \cap N_j)+1$. So $Y \subseteq E_j$ because $[p_{N_j}](Y)$ is a point.
\end{proof}

Now we are ready to prove Theorem~\ref{ThmBettiRankFormulaMain} and Theorem~\ref{ThmZariskiClosedDegLoci}.
\begin{proof}[Proof of Theorem~\ref{ThmBettiRankFormulaMain}]
Let $S \subseteq \sM_g$ be a subvariety. 
Assume $S^{\mathrm{Betti}}(1) = S$. Then Lemma~\ref{PropFoliationDegLociSubset} implies that $\bigcup_{j=1}^k E_j = S$. But each $E_j$ is Zariski closed in $\sM_g$. Hence $\sM_g = E_j$ for some $j \in \{1,\ldots,k\}$. Up to renumbering we may assume $S = E_1$. Now we have the Mumford--Tate domain $\sD_1 = \mathbf{G}_1(\BR)^+x_1$, and a normal subgroup $N_1\lhd \mathbf{G}_1$. By definition, $E_1 \subseteq u(\sD_1)\cap S$, so $S \subseteq u(\sD_1)$.

Now let us consider the case $S = \sM_g$. Then  $\mathbf{G}_1 = V\rtimes \mathrm{GSp}_{2g}$ since $\sM_g \subseteq u(\sD_1)$. The normal subgroup $N_1$ can be chosen such that the reductive part is semi-simple, \textit{i.e.} $N_1 \lhd V\rtimes \mathrm{Sp}_{2g}$; see the statement of \cite[Thm.~7.1]{BaldiUrbanik}. 
Since $V$ is irreducible as an $\mathrm{Sp}_{2g}$-module by Proposition~\ref{PropGSNFSubVHS}.(iii), the only normal subgroups of $V\rtimes \mathrm{Sp}_{2g}$ are $0$, $V$, and $V\rtimes \mathrm{Sp}_{2g}$.

If $N_1 = 0$, then $[p_{N_1}]$ is identity and $E_1 = \{s \in\sM_g(\BC): 0 \ge 1\}$ is empty.

If $N_1 = V$, then each fiber of $[p_{N_1}]$ has $\BC$-dimension $\frac{1}{2}\dim_{\BQ}  V$. So 
$$E_1 \subseteq \{s\in \sM_g(\BC): \frac{1}{2}\dim_{\BQ}  V \ge \frac{1}{2}\dim_{\BQ}  V + 1\}$$ is empty.

If $N_1 = V\rtimes \mathrm{Sp}_{2g}$, then $\sD_1/N_1$ is a point. So 
\begin{align*}
E_1 =&\left \{ s \in \sM_g(\BC) : \dim \sM_g \ge \frac{1}{2}\dim_{\BQ}  V + 1\right\}\\
 =&\left \{s\in \sM_g(\BC) : 3g-3 \ge \frac{1}{2}\binom{2g}{3} - \frac{1}{2}g + 1\right\},
 \end{align*}
  which is again empty for $g \ge 3$.

Thus we get a contradiction to $\sM_g = E_1$. So $\sM_g^{\mathrm{Betti}}(1) \not= \sM_g$. So $\beta_{\GS}^{\wedge (3g-3)} \not\equiv 0$ by Proposition~\ref{CorBettiFormBettiRank}. We are done.
\end{proof}

\begin{proof}[Proof of Theorem~\ref{ThmZariskiClosedDegLoci}]
By Lemma~\ref{PropFoliationDegLociSubset} and because each $E_j$ is Zariski closed in $S$, it suffices to prove that
\begin{equation}
\bigcup\nolimits_{j=1}^k E_j \subseteq S^{\mathrm{Betti}}(1).
\end{equation}

Now let us prove $E_j \subseteq S^{\mathrm{Betti}}(1)$. To ease notation, use $E$ (resp. $\sD$, $N$) to denote $E_j$ (resp. $\sD_j$, $N_j$).

For each $s \in E$, by definition there exists a complex analytic irreducible $$\wt{Y} \subseteq u^{-1}(S) \cap \sD$$ such that $s \in u(\wt{Y})$, $\dim \wt{Y} \ge \frac{1}{2}\dim_{\BQ}  (V \cap N)+1$, and that $\wt{Y}$ is contained in a fiber of the quotient $\sD \rightarrow \sD/N$. The last condition implies that $\wt{Y} \subseteq N(\BR)^+x$ for some $x \in \sD$. 

Under the identification $\CM = V(\BR) \times \CM_0$, by Lemma~\ref{LemmaWeakMTDomain} we have 
$$N(\BR)^+x = (v+V_N(\BR)) \times p(\sD)$$ where $V_N = V \cap N$ and $v \in V(\BR)$, for the projection $p \colon \CM \rightarrow \CM_0$. Now 
for each $$(a,x_0) \in \wt{Y} \subseteq (v+V_N(\BR)) \times \CM_0,$$ there exists a complex analytic curve $\wt{C} \subseteq \CM_0$  such that $\{a\} \times \wt{C} \subseteq \wt{Y}$ because $\dim  \wt{Y} \ge \frac{1}{2}\dim_{\BQ}  V_N+1$. Hence $s \in S^{\mathrm{Betti}}(1)$ by Lemma~\ref{LemmaBettiFoliationBettiRank}. This proves $E \subseteq S^{\mathrm{Betti}}(1)$, and we are done.
\end{proof}

\section{Proof of main theorem}\label{SubsectionBignessProof}
\subsection{Proof of the height comparison}
By \cite[Thm.~5.3.5]{YZ21}, lower bounds of the height function $h_{\overline{\CL}}$ correspond to  bigness properties of $\overline{\CL}$. In this paper, we prove the bigness of the {\em generic fiber} (or called the {\it geometric part}) $\wt{\CL}$ of $\overline{\CL}$ (see \cite[$\S$1.1.5]{YZ21} for the definition) and deduce the desired height comparison from it. 

\begin{thm}\label{ThmBigAdeLB}
Assume $g \ge 3$. Then the adelic line bundle $\wt{\CL}$ is big, \textit{i.e.} $\widehat{\mathrm{vol}}(\wt{\CL}) > 0$.
\end{thm}
\begin{proof}
By the flatness of extension $\BQ\subseteq \BC$, we have $\widehat{\mathrm{vol}}(\wt\CL)=\widehat{\mathrm{vol}} (\wt\CL_\BC)$. 
So $$\widehat{\mathrm{vol}}(\wt{\CL}) =  \int_{\sM_g(\BC)}\beta_{\GS}^{\wedge (3g-3)}$$ by Theorem~\ref{PropHSforGS}. Now $\beta_{\GS}^{\wedge (3g-3)} \ge 0$ since $\beta_{\GS}$ is semi-positive, and $\beta_{\GS} ^{\wedge (3g-3)} \not\equiv 0$ by Theorem~\ref{ThmBettiRankFormulaMain}. So $\widehat{\mathrm{vol}}(\wt{\CL}) > 0$.
\end{proof}

Now, we are ready to prove our main theorem.
\begin{proof}[Proof of Theorem~\ref{MainTheorem}]
It suffices to prove the height comparison \eqref{EqGSComparison} for $\pair{\GS(s),  \GS(s)}_{\mathrm{BB}}$ because, by work of the second-named author \cite{Zh10}, the Ceresa cycles and Gross–Schoen have the same height up to some positive multiple. 
Next by \cite[Cor.~2.5.2]{Zh10}, it suffices to prove  \eqref{EqGSComparison} for $e_s = \xi_s$, \textit{i.e.} there exists a Zariski open dense subset $U$ of $\sM_g$ defined over $\BQ$ and positive numbers $\epsilon$ and $c$ such that
\begin{equation}\label{EqGSComparison2}
\pair{\GS(\xi_s),  \GS(\xi_s)}_{\mathrm{BB}} \ge \epsilon h_{\Fal}(C_s) - c
\end{equation}
for all $s \in U(\IQbar)$.  
The conclusion follows immediately from Theorem~\ref{ThmBigAdeLB} and \cite[Thm.~5.3.5.(iii)]{YZ21}. 
 Notice that one can either take the adelic line bundle $\overline{\CM}$ on $\sM_g$ which defines the Faltings height (it exists by \cite[$\S$2.6.2]{YZ21}), or one can take any ample line bundle on a suitable compactification of $\sM_g$ and then use the comparison of the logarithmic Weil height and the Faltings height.
\end{proof}

\subsection{Description of the ample locus $\sM_g^{\mathrm{amp}}$}

We start with the following proposition, which is a consequence of our volume identity.
\begin{prop}\label{cor-rationality}
Let $S$ be a subvariety of $\sM_g$ defined over a subfield $k$ of $\BC$. 
Then $S^{\mathrm{Betti}} (1)$ is also defined over $k$. In particular $\sM_g^{\mathrm{Betti}}(1)$ is defined over $\BQ$.
\end{prop}
\begin{proof}
The definition of the volume is
\[
\widehat{\mathrm{vol}}(\wt{\CL}_{\BC}|_S) := \lim_{m\rightarrow \infty}\frac{(\dim S)!}{m^{\dim S}} \hat{h}^0(S,\wt{\CL}_{\BC}|_S^{\otimes m})
\]
where $\hat{h}^0(S,\wt{\CL}_{\BC}|_S^{\otimes m})$ is the dimension of the space of effective sections of $\wt{\CL}_{\BC}|_S^{\otimes m}$. For a precise definition, we refer to \cite[Defn.~5.1.3 and Thm.~5.2.1]{YZ21}.

Let $S'$ be an irreducible component of  $S^{\mathrm{Betti}} (1)$. Then $S' = S^{\prime \mathrm{Betti}}(1)$, and hence $\beta_{\GS}^{\wedge \dim S'} \equiv 0$ on $S'$ by Corollary~\ref{CorBettiFormBettiRank}. So  $\widehat{\mathrm{vol}}(\wt{\CL}_{\BC}|_{S'})=0$ by Theorem~\ref{PropHSforGS}. 
 For any automorphism $\sigma\in \Aut (\BC/k)$, let $S^{\prime\sigma}$ denote $S'\otimes _\sigma \BC$. Then $S^{\prime\sigma}$ is also a closed subvariety of $S_{\BC}$. Moreover $\dim S^{\prime\sigma}=\dim S'$ and $\hat{h}^0(S',\wt{\CL}_{\BC}|_{S'}^{\otimes m}) = \hat{h}^0(S^{\prime\sigma},\wt{\CL}_{\BC}|_{S^{\prime\sigma}}^{\otimes m})$. So 
 
 $$\widehat{\mathrm{vol}}(\wt\CL_{\BC}\otimes _\sigma \BC|_{S'\otimes _\sigma \BC})
=\widehat{\mathrm{vol}}(\wt{\CL}_{\BC}|_{S'})=0.$$
 
As $\wt\CL$ is defined over $\BQ$, Theorem~\ref{PropHSforGS} implies
$$\int_{S^{\prime\sigma}(\BC)}c_1(\overline{\CL}_{\BC})^{\wedge \dim S'}=0.$$
Thus $\beta_{\GS}^{\wedge \dim S'} \equiv 0$ on $S^{\prime\sigma}$ by Proposition~\ref{PropTwoCurvatureEqualAdeGS}, and therefore
$$S^{\prime\sigma} (\BC)=(S^{\prime\sigma})^{\mathrm{Betti}} (1)\subseteq S^{\mathrm{Betti}} (1).$$
This shows that $S^{\mathrm{Betti}}(1) \otimes _\sigma \BC=S^{\mathrm{Betti}}(1)$.
 Thus $S^{\mathrm{Betti}} (1)$ is defined over $k$.
\end{proof}
Now we are ready to prove our description of $\sM_g^{\mathrm{amp}}$ as in Remark~\ref{RemarkMainTheorem}. More precisely

\begin{thm}\label{ThmDescriptionAmpleLocus}
We have $\sM_g^\amp=\sM_g\setminus \sM_g^{\mathrm{Betti}}(1)$. Moreover $\GS(s)$ is non-torsion in $\mathrm{Ch}^2(C_s^3)$ for any $s \in \sM_g^\amp(\BC) \setminus \sM_g^\amp(\IQbar)$.
\end{thm}
\begin{proof}
The height inequality \eqref{EqGSComparison} is invariant under the $\mathrm{Gal}(\IQbar/\BQ)$-action. Hence $\sM_g^{\mathrm{amp}}$ is defined over $\BQ$.


We start by showing $\sM_g^{\mathrm{amp}} \supseteq \sM_g\setminus \sM_g^{\mathrm{Betti}}(1)$, or equivalently $\sM_g \setminus \sM_g^{\mathrm{amp}} \subseteq \sM_g^{\mathrm{Betti}}(1)$. 
By definition of $\sM_g^{\mathrm{amp}}$, the height inequality \eqref{EqGSComparison2} cannot hold for any non-empty open subvariety $V$ of $\sM_g \setminus \sM_g^{\mathrm{amp}}$. So by \cite[Thm.~5.3.7]{YZ21}, $\wt{\CL}$ cannot be big on any irreducible component $S$ of $(\sM_g \setminus \sM_g^{\mathrm{amp}})_{\IQbar}$, and hence $\beta_{\GS}|_{S(\BC)}^{\wedge \dim S}\equiv 0$ by the volume identity Theorem~\ref{PropHSforGS}. Thus $S_{\BC} = S_{\BC}^{\mathrm{Betti}}(1)$ by Proposition~\ref{CorBettiFormBettiRank}, and therefore is contained in  $\sM_{g,\BC}^{\mathrm{Betti}}(1)$. So $\sM_g \setminus \sM_g^{\mathrm{amp}} \subseteq \sM_{g,\BC}^{\mathrm{Betti}}(1)$. This proves $\sM_g \setminus \sM_g^{\mathrm{amp}} \subseteq \sM_g^{\mathrm{Betti}}(1)$ by Proposition~\ref{cor-rationality}. 

Next let us prove that $S:= \sM_g^{\mathrm{Betti}}(1) \cap \sM_g^\amp$ is empty. By definition of $\sM_g^\amp$, the inequality \eqref{EqGSComparison2} holds for all $s \in S(\IQbar)$. Assume $S\not=\emptyset$. Since $S \subseteq \sM_g^\amp$, we have that $\wt{\CL}|_Z$ is big by \cite[Thm.~5.3.8]{YZ21} for some irreducible component $Z$ of $S$, \textit{i.e} $\widehat{\mathrm{vol}}(\wt{\CL}|_Z) > 0$. But $Z \subseteq S \subseteq \sM_g^{\mathrm{Betti}}(1)$, so $\beta_{\GS}|_{Z(\BC)}^{\wedge\dim Z} \equiv 0$ by Proposition~\ref{CorBettiFormBettiRank}, and hence $\widehat{\mathrm{vol}}(\wt{\CL}|_Z) = 0$ by the volume identity Theorem~\ref{PropHSforGS}. This is a contradiction. So $S = \emptyset$.

Now we can conclude $\sM_g^\amp=\sM_g\setminus \sM_g^{\mathrm{Betti}}(1)$ by the last two paragraphs.

Take $s \in \sM_g(\BC) \setminus \sM_g(\IQbar)$, and let $\bar{s}$ be the $\IQbar$-Zariski closure of $s$ in $\sM_g$. Then $\dim \bar{s} \ge 1$. 
Assume $\GS(s)$ is torsion in $ \mathrm{Ch}^2(C_s^3)$. Then $\GS(t)$ is a torsion point of $ \mathrm{Ch}^2(C_t^3)$ for each $t \in \bar{s}(\IQbar)$, and hence $t\mapsto \AJ(\GS(t))$ is a torsion section of $\sJ^2(\sC_g^{[3]}/\sM_g)\times_{\sM_g} \bar{s} \rightarrow \bar{s}$.  So $\nu_{\GS}(\bar{s}) \subseteq \CF_{\mathrm{Betti}}$ by definition of the Betti foliation because $\dim \bar{s} \ge 1$. So $\bar{s} \subseteq \sM_g^{\mathrm{Betti}}(1)$ by definition \eqref{EqDefnBettiStratum}. We are done
\end{proof}

\section{A discussion on general families}\label{SectionFuture}
In this section, we explain some conjectures on the rationality of the Betti strata, the height inequality, and the volume identity.

Let $S$ be a quasi-projective complex variety,  $f \colon X \rightarrow S$ be a smooth family of projective varieties, and $N$ be a relatively ample line bundle on $X$. Let $Z$ be a family of homologically trivial $n$-cocycles in $X/S$ with $n\le \frac {d+1}2$, which is homologically trivial and primitive on each geometric fiber 
$X_s$ in the sense that $c_1(N_s)^{d+2-2n}\cdot Z_s=0$ in $\Ch^{d+2-n}(X_s)$. 

\subsection{Betti strata}
We have a normal function $\nu \colon S \rightarrow \sJ ^n(X/S)$ as defined by  \eqref{EqNormalFunctionFromFamily}, where $\sJ^n(X/S)$ is the intermediate Jacobian associate with the following VHS:
$$\BH^n:=\ker \left(c_1(N)^{d+2-2n} \colon R^{2n-1}f_*\BZ (n)\lra R^{2d-2n+3}f_*\BZ(d+2-n)\right).$$

We can define the Betti stratum $S^{\mathrm{Betti}} (1)$ as in \eqref{EqDefnBettiStratum}. The proof of Theorem~\ref{ThmZariskiClosedDegLoci} can be adapted to prove:
\begin{thm}\label{ThmBettiStrataZarClosedGeneralFamily}
The Betti stratum $S^{\mathrm{Betti}}(1)$ is Zariski closed in $S$.
\end{thm}
A detailed proof can be found in \cite{GZSurvey} (as for Theorem~\ref{ThmBettiRankGeneralFamily}). 
We propose the following conjecture on the rationality of the Betti strata.
\begin{conj} \label{conj-rationality}
Let $k$ be a subfield of $\BC$ over which $f$ is defined. Then $S^{\mathrm{Betti}} (1)$ is defined over $k$.
\end{conj} 
In this paper, we proved Conjecture~\ref{conj-rationality} for the Gross--Schoen cycles and Ceresa cycles for any subvariety $S$ of $\sM_g$ by a volume identity \eqref{EqVolume}; see Proposition~\ref{cor-rationality}. We expect this volume identity to be true for general family $f \colon X\rightarrow S$. 

\subsection{Height inequality}
Now, we assume that $f$ and $N$ are defined over a number field $K$; in general, any family $X/S$ can be embedded into one defined over a number field $K$ by the Lefschetz principle. {\em Assume that the Beilinson--Bloch pairing is well-defined for the cycle $Z_s$ for 
each $s\in S(\IQbar)$.} This means that $Z_s$ can be extended to a cycle $\CZ_s$ on an integral model $\CX_s$ of $X_s$ over $O_{K(s)}$ 
such that the restrictions of $\CZ_s$ on special fibers are homologically trivial. Let $g_s$ be the Green's current on $X_s(\BC)$ of degree $(n-1, n-1)$ so that 
$\pp g_s=\delta _{Z_s(\BC)}$. Then we have an arithmetic cycle $\overline \CZ_s=(\CZ_s, g_s)$. The Beilinson--Bloch height  of $Z_s$ is defined in terms of arithmetic intersection number on $\CX_s$:
\begin{equation}
\label{eq-BB}
\pair{Z_s, Z_s}_{\mathrm{BB}}=(-1)^i \wh c_1(\overline \CN_s)^{d+1-2n}\overline \CZ_s\cdot \overline \CZ_s,
\end{equation}
where $\overline \CN_s$ is a Hermitian line bundle on $\CX_s$  extending  $N_s$.
They even conjectured that this pairing is non-negative and is zero if and only if $Z_s$ is torsion in the Chow group. 
See \cite{Be, Bl, GS91, GS94, Zh22} for details and various connections to the Grothendieck standard conjectures. 
Besides Gross--Schoen \cite{GS95} for the triple product of curves, we know that the height pairing is unconditionally defined on abelian varieties by K\"unnemann \cite{Ku}, and 
the product of a curve and a surface by the second-named author \cite{Zh24}.

In the following, we conjecture a height inequality as in our main Theorem~\ref{MainTheorem}.
We say that a function $\mathbf{h} \colon S(\IQbar) \rightarrow \BR$ is a {\em dominant height on an open subvariety $U$ of $S$ over $K$}, 
if there exist a immersion $S \subseteq \BP^N$ defined over $\IQbar$ 
 and positive numbers $\epsilon$ and $c$ such that 
\begin{equation}\label{EqComparison}
\mathbf{h}(s) \ge \epsilon h_{\mathrm{Weil}}(s)-c \qquad \text{ for all }s \in U(\IQbar),
\end{equation}
where $h_{\mathrm{Weil}}$ is the logarithmic Weil height on $\BP^N(\IQbar)$. 
Notice that if this condition is satisfied, then the height inequality  \eqref{EqComparison} remains true with $h_{\mathrm{Weil}}$ replaced by $h_{\overline M}$ for any integrable adelic line bundle $\overline M$ on $S$   up to changing $\epsilon>0$ and $c>0$. 

Write 
\[
S^\amp=S\setminus S^{\mathrm{Betti}} (1)
\]
 with respect to the normal function $\nu$ constructed from $Z$. We conjecture:

\begin{conj}\label{conj-inequality}
The Beilinson--Bloch height function $s \mapsto \pair{Z_s, Z_s}_{\mathrm{BB}}$ is a dominant height on $S^\amp$. Moreover, $S^\amp$ is the largest open subvariety of $S$ on which $s \mapsto \pair{Z_s, Z_s}_{\mathrm{BB}}$ is a dominant height. 
\end{conj}
In practice, in addition to this conjecture one also needs to check whether $S^\amp$ is empty, for which we use Theorem~\ref{ThmBettiRankGeneralFamily} below.
The second part of Conjecture~\ref{conj-inequality} is equivalent to: For any immersion $S \subseteq \BP^N$ defined over $\IQbar$, there exists a Zariski dense  
sequence $s_i\in (S\setminus S^\amp) (\IQbar)$ such that 
$$\lim _i h_{\mathrm{Weil}} (s_i)=\infty, \qquad \lim _i \frac {\pair{Z_{s_i}, Z_{s_i}}_{\mathrm{BB}}}{h_{\mathrm{Weil}} (s_i)}=0.$$

The proof of Theorem~\ref{ThmDescriptionAmpleLocus} in fact proves the following:
\begin{thm}\label{ThmAmpleLocusConjectureGS}
Conjecture~\ref{conj-inequality} holds in the case of Gross--Schoen and Ceresa cycles for any subvariety of $\sM_g$.
\end{thm}
A key ingredient to prove Theorem~\ref{ThmAmpleLocusConjectureGS} is the volume identity  \eqref{EqVolume}.  The same argument can be used to show that 
 the conjectured general volume identity \eqref{eq-volume} implies the full Conjecture~\ref{conj-inequality} with the help of \cite[Thm.~5.3.8]{YZ21}.

In the rest of this subsection let us explain the criterion to check $S^\amp \not=\emptyset$. As in $\S$\ref{SubsectionPeriodMap}, the normal function $\nu$ gives rise to a period map $\varphi_{\nu} \colon S \rightarrow \Gamma\backslash\CM$ where $\CM$ is a suitable parametrizing space of $\BQ$-mixed Hodge structures of weight $-1$ and $0$, and we can define the generic Mumford--Tate group $\mathrm{MT}(\nu(S))$ as at the end of $\S$\ref{SubsectionPeriodMap}.

\begin{thm}\label{ThmBettiRankGeneralFamily}
The followings are equivalent:
\begin{enumerate}
\item[(i)] $S^\amp \not=\emptyset$;
\item[(ii)] For any normal subgroup $N$ of the generic Mumford--Tate group $\mathrm{MT}(\nu(S))$, we have 
\begin{equation}\label{EqTrivialUpperBoundBettiRankBla}
\dim S \le \dim [p_N](\varphi_{\nu}(S)) + \frac{1}{2}\dim_{\BQ} (V \cap N).
\end{equation}
\end{enumerate}
\end{thm}
The proof is an adaption of our proof of Theorem~\ref{ThmBettiRankFormulaMain} executed in $\S$\ref{SectionBettiRank}. Let us give a sketch here. More details will be given in \cite{GZSurvey} where we prove a formula for the Betti rank $r(\nu)$ defined in \eqref{EqDefnBettiRank}.

Observe the easy upper bound  $r(\nu) \le \min\{ \dim S, \frac{1}{2}\dim_{\BQ} V\}$ by the description of the leaves of $\CF_{\mathrm{Betti}}$ below \eqref{EqPeriodJacobian}. This trivial upper bound can be easily generalized to: $r(\nu) \le $ the right hand side of \eqref{EqTrivialUpperBoundBettiRankBla} for any such $N$. 

Let us prove (i)$\Rightarrow$(ii). (i) implies $S \not= S^{\mathrm{Betti}}(1)$, and hence $r(\nu) = \dim S$ by Proposition~\ref{CorBettiFormBettiRank}. So (ii) holds by the last sentence of the previous paragraph.

To prove (ii)$\Rightarrow$(i), replace $S$ by $\varphi_{\nu}(S)$ (which is still a quasi-projective complex variety by the main result of \cite{BBTmixed}) and assume for contradictory $S = S^{\mathrm{Betti}}(1)$. All arguments in $\S$\ref{SectionBettiRank} can be applied to our $S$ until Lemma~\ref{PropFoliationDegLociSubset}. 
Then $S = E_1$, each fiber of $[p_{N_1}]|_S$ has $\BC$-dimension $\ge \frac{1}{2}\dim_{\BQ}(V\cap N) +1$ by definition of $E_1$. So $\dim S \ge \dim [p_{N_1}](S) +  \frac{1}{2}\dim_{\BQ}(V\cap N_1) +1$ by the Fiber Dimension Theorem. So $\dim S > $ the right hand side of \eqref{EqTrivialUpperBoundBettiRankBla} with $N = N_1$, contradiction to (ii). 

\subsection{Adelic line bundle and volume identity}
In \cite{Blo2}, Bloch constructed a biextension line bundle $L$ on $S$ defined over $K$. At each closed point $s\in S(\IQbar)$ where the Beilinson--Bloch height is well-defined, this line bundle has an adelic structure $\overline L_s$ over $K(s)$ whose degree gives the Beilinson--Bloch height \eqref{eq-BB}.  The archimedean part of this adelic structure has been studied by  Hain \cite{Ha90}. 
\begin{conj}\label{conj-adL}
The line bundle $L$ has a natural extension to a nef adelic line bundle $\overline L$ on $S$, such that its restriction to each $s\in S(\IQbar)$ where the Beilinson--Bloch height is well-defined is given by $\overline L_s$.
\end{conj}
By the arithmetic Hilbert--Samuel formula \cite[Thm.~5.2.2]{YZ21}, one consequence of this conjecture is the following volume identity for the generic fiber $\wt L$ of $\overline L$: For any archimedean place $v$ of $K$ and any subvariety $S'$ of $S_{v,\BC}$, we have
\begin{equation}\label{eq-volume}
\wh \vol (\wt L|_{S'})=\int _{S'(\BC)}c_1(\bar L_v)^{\wedge \dim S'}.
\end{equation}

For the Gross--Schoen and Ceresa cycles for any subvariety of $\sM_g$, we prove the existence of the extension in $\S$\ref{SubsectionConstrALB} and the volume identity as Theorem~\ref{PropHSforGS}, but not the nefness of the adelic line bundle in question.

\appendix
\renewcommand{\thesection}{\Alph{section}}
\setcounter{section}{0}

\section{Mixed Hodge structures and variations}\label{AppVMHS}
This Appendix collects some definitions about variations of mixed Hodge structures.
\subsection{Definitions}
Let $R$ be a subring of $\BR$.
\begin{defn}[pure Hodge structures]
Let $V$ be a free $R$-module of finite rank. An {\it $R$-pure Hodge structure} on $V$ of weight $n$ is a bigrading $V_{\BC} = \bigoplus_{p+q=n} V^{p,q}$ such that $V^{q,p} = \overline{V^{p,q}}$ for all $p,q$.
\end{defn}
Pure Hodge structures can also be defined using the (decreasing) {\it Hodge filtration} $F^p V_{\BC} := \bigoplus_{r\ge p}V^{r,n-r}$; then the bigrading  is uniquely determined by $V^{p,n-p} = F^p V_{\BC} \cap \overline{F^{p+1} V_{\BC} }$.

\begin{defn}[mixed Hodge structures]  Let $M$ be a free $R$-module of finite rank. An {\it $R$-mixed Hodge structure} on $M$ is a triple $(M,W_{\bullet},F^{\bullet})$ consisting of two filtrations, an increasing filtration $W_\bullet$ on $M_{\BQ}$ (the {\em weight filtration}) and a decreasing filtration $F^\bullet$ on $M_{\BC}$ (the {\em Hodge filtration}), such that for each $k \in \BZ$, $\mathrm{Gr}_k^W M_{\BQ} = W_k/W_{k-1}$ is a $\BQ$-pure Hodge structure of weight $k$ for the filtration on $\mathrm{Gr}_k^W (M_\BC)$ deduced from $F^\bullet$. 
\end{defn}

For a mixed Hodge structure $(M,W_{\bullet},F^{\bullet})$, the numbers $k \in \BZ$ such that $\mathrm{Gr}_k^W M_{\BQ} \not=0$ are called its {\em weights}, and  
the numbers
$ h^{p,q}(M) = \dim_{\BC}
F^p \mathrm{Gr}^W_{p+q}(M_{\BC}) / F^{p+1} \mathrm{Gr}^W_{p+q}(M_{\BC})$
are called its {\em Hodge numbers} .


For each $n \in \BZ$, define $R(n)$ to be the pure Hodge structure on $R$ of weight $-2n$ such that $R(n)^{-n,-n} = \BC$ and $R(n)^{p,q}= 0$ for all $(p,q)\not=(-n,-n)$. 

A {\em polarization} on a pure Hodge structure $V$ of weight $n$ is a morphism of Hodge structures 
\[
Q \colon V_{\BQ} \otimes V_{\BQ} \longrightarrow \BQ(-n)
\]
such that 
 the Hermitian form on $V_{\BC}$ given by $Q(Cu,\bar{v})$ is positive-definite where $C$ is the Weil operator ($C|_{H^{p,q}} = i^{p-q}$ for all $p,q$).



\subsection{Mumford--Tate group}\label{SubsectionMTGroup}
Now, let us turn to a more group theoretical point of view on mixed Hodge structures. Let $\BS = \mathrm{Res}_{\BC/\BR}\BG_{\mathrm{m},\BC}$ be the Deligne torus, \textit{i.e.} the real algebraic group such that $\BS(\BR) = \BC^*$ and $\BS(\BC) = \BC^* \times \BC^*$, and that the complex conjugation on $\BS(\BC)$ sends $(z_1,z_2) \mapsto (\bar{z}_2, \bar{z}_1)$.

As for pure Hodge structures, mixed Hodge structures can also be equivalently defined in terms of {\em bigradings} by Deligne \cite[1.2]{PinkThesis}.  
Given a $\BQ$-vector space $M$ of finite dimension, a bigrading $M_{\BC} = \oplus_{p, q \in \BZ} I^{p,q}$ is  equivalent to 
a homomorphism $h \colon \BS_{\BC} \rightarrow \mathrm{GL}(M_{\BC})$. In particular, any mixed Hodge structure on $M$ defines a unique homomorphism $h \colon \BS_{\BC} \rightarrow \mathrm{GL}(M_{\BC})$ and we use $(M,h)$ to denote this mixed Hodge structure.

\begin{defn}
For any mixed Hodge structure $(M,h)$, its {\em Mumford--Tate group} is the smallest $\BQ$-subgroup $\mathbf{G}$ of $\mathrm{GL}(M_{\BQ})$ such that $h(\BS_{\BC}) \subseteq \mathbf{G}(\BC)$.
\end{defn}

\begin{lemma}\label{LemmaMTGroup}
Assume $(M,h)$ has weight $-1$ and $0$. Then $h$ is defined over $\BR$. 
\end{lemma}
\begin{proof}
Deligne's bigrading satisfies: $W_n M_{\BC} = \oplus_{p+q\le n}I^{p,q}$, $F^m M_{\BC} = \oplus_{p\ge m}I^{p,q}$, $I^{p,q} \equiv \overline{I^{q,p}} \bmod \bigoplus_{r<p,s<q} I^{r,s}$. To prove that $h$ is defined over $\BR$, it suffices to prove that $I^{p,q} = \overline{I^{q,p}}$ for all $p,q$.

By our assumption, $I^{p,q}\not= 0$ only if $p+q = -1$ or $0$. For $r<p$ and $s<q$, we have $r\le p-1$ and $s\le q-1$, and hence $r+s \le p+q-2$. So for any $p+q = 0$ or $-1$, we have $\bigoplus_{r<p,s<q}I^{r,s} \subseteq \bigoplus_{r+s\le -2} I^{r,s} = 0$. We are done. 
\end{proof}

\subsection{Variation of mixed Hodge structures} 
\begin{defn}\label{DefnVMHS}
Let $S$ be a connected complex manifold. A {\em variation of mixed Hodge structures (VMHS)} on $S$  is a triple $(\BM_{\BZ}, W_{\bullet}, \sF^{\bullet})$ consisting of
\begin{itemize}
\item[-] a local system $\BM_{\BZ}$ of  free $\BZ$-modules of finite rank on $S$ ,
\item[-] a finite increasing filtration ({\em weight filtration}) $W_{\bullet}$ of the local system $\BM := \BM_{\BZ}\otimes_{\BZ_S}\BQ_S$ by local subsystems,
\item[-] a finite decreasing filtration ({\em Hodge filtration}) $\sF^{\bullet}$ of the holomorphic vector bundle $\sM:= \BM_{\BZ} \otimes_{\BZ_S} \sO_S$  by holomorphic subbundles 
 \end{itemize}
 satisfying the following properties
 \begin{enumerate}
 \item[(i)] for each $s \in S$, the triple $(\BM_s, W_{\bullet}, \sF^{\bullet}_s)$ defines a mixed Hodge structure on $\BM_s$,
 \item[(ii)] the connection $\nabla \colon \sM \rightarrow \sM \otimes_{\sO_S}\Omega_S^1$ whose sheaf of horizontal sections is $\BM_{\BC}:=\BM_{\BZ}\otimes_{\BZ_S}\BC_S$ satisfies the {\em Griffiths' transversality} condition
 \[
 \nabla(\sF^p)\subseteq \sF^{p-1}\otimes\Omega_S^1.
 \]
 \end{enumerate}
\end{defn}
\noindent The weights and Hodge numbers of $(\BM_s, W_{\bullet}, \sF^{\bullet}_s)$ are the same for all $s \in S$. We call them the {\em weights} and the {\em Hodge numbers} of the VMHS $(\BM_{\BZ}, W_{\bullet}, \sF^{\bullet})$.

\smallskip
If there is only one $n\in \BZ$ such that $\mathrm{Gr}^W_n\BM \not=0$, then each fiber of this VMHS is a pure Hodge structure of weight $n$. In this case, the VMHS is said to be {\em pure}. More precisely, we have the following definition.
\begin{defn}
A {\em variation of Hodge structures (VHS) of weight $n$} on $S$ is a pair $(\BV_{\BZ},\sF^{\bullet})$ such that  $(\BV_{\BZ}, W_{\bullet}, \sF^{\bullet})$ is a VMHS, where $W_{\bullet}$ is the increasing filtration on $\BV:= \BV_{\BZ}\otimes_{\BZ_S}\BQ_S$ defined by $W_{n-1} =0$ and $W_n = \BV$.
\end{defn}
\noindent To each VMHS $(\BM_{\BZ}, W_{\bullet}, \sF^{\bullet})$ on $S$, we can associate variations of pure Hodge structures obtained from $\mathrm{Gr}_k^W\BM$.

 We shall use the following convention: For each $n \in \BZ$ and any VHS $(\BV_{\BZ},\sF^{\bullet}) \rightarrow S$, define $\BV_{\BZ}(n)$ to be the VHS $(\BV_{\BZ},\sF^{\bullet -n})$. In particular $\BZ(n)_S$ be the VHS on $S$ of weight $-2n$ such that $(\BZ(n)_S)_s = \BZ(n)$ for each $s \in S$. 

\begin{defn}
A {\em polarization} of VHS $(\BV_{\BZ},\sF^{\bullet})$ on $S$ of weight $n$ is a morphism of VHS $\BV \otimes \BV \rightarrow \BQ(-n)_S$ inducing on each fiber a polarization.

 We say that a VMHS $(\BM_{\BZ}, W_{\bullet}, \sF^{\bullet})$ is {\em graded-polarizable} if $\mathrm{Gr}_k^W\BM$ has a  polarization for each $k \in \BZ$.
\end{defn}

\def\cprime{$'$}


\vfill

\end{document}